\documentclass[a4paper,12pt,oneside,reqno]{amsart}

\usepackage[T1]{fontenc}
\usepackage[utf8]{inputenc}
\usepackage{lmodern}
\usepackage[english]{babel}

\usepackage[%
	left=2.5cm,       
	right=2.5cm,      
	top=3.5cm,        
	bottom=3.5cm,     
	heightrounded,    
	bindingoffset=0mm 
]{geometry}

\usepackage{amsmath}
\usepackage{amssymb}
\usepackage{mathtools}
\usepackage{mathrsfs}

\usepackage[
colorlinks=true,
urlcolor=teal,
citecolor=magenta,
linkcolor=teal,
breaklinks=true]
{hyperref}

\usepackage[noabbrev,capitalize,nameinlink]{cleveref}

\usepackage{bookmark}
\usepackage{calrsfs}

\usepackage[initials,msc-links]{amsrefs}

\usepackage{graphicx}
\usepackage{xcolor}
\usepackage{enumitem}
\usepackage{url}

\numberwithin{equation}{section}

\theoremstyle{plain}
\newtheorem{theorem}{Theorem}[section]
\newtheorem{lemma}[theorem]{Lemma}
\newtheorem{proposition}[theorem]{Proposition}
\newtheorem{corollary}[theorem]{Corollary}

\theoremstyle{definition}
\newtheorem{definition}[theorem]{Definition}
	
\newcommand{\N}{\mathbb{N}}
\newcommand{\R}{\mathbb{R}}
\newcommand{\e}{\mathrm{e}}
\newcommand{\eps}{\varepsilon}
\newcommand{\de}{\partial}
\newcommand{\weakto}{\rightharpoonup}
\newcommand{\weakstarto}{\stackrel{*}{\rightharpoonup}}
\newcommand{\Haus}[1]{\mathscr{H}^{#1}}
\newcommand{\Leb}[1]{\mathscr{L}^{#1}}
\newcommand{\redb}{\mathscr{F}} 
\newcommand{\loc}{{\mathrm{loc}}}
\newcommand{\NL}{{\mathrm{NL}}}
\newcommand{\M}{\mathscr{M}}
\newcommand{\di}{{\,\mathrm{d}}}
\newcommand{\sing}{{\mathrm{s}}}
\newcommand{\ass}{\mathrm{ac}}
\newcommand{\eff}{{\mathrm{e}}}
\newcommand{\rap}{\mathscr{R}}

\renewcommand{\phi}{\varphi}
\renewcommand{\rho}{\varrho}
\renewcommand{\theta}{\vartheta}
\renewcommand{\div}{\mathrm{div}} 

\DeclareMathOperator{\supp}{supp}
\DeclareMathOperator{\tang}{Tan}
\DeclareMathOperator{\Lip}{Lip}
\DeclareMathOperator{\esssup}{ess\,sup}
\DeclareMathOperator{\dimension}{dim}
    
\DeclarePairedDelimiter{\set}{\{}{\}}

\newcommand{\mres}{\mathbin{\vrule height 1.6ex depth 0pt width
0.13ex\vrule height 0.13ex depth 0pt width 1.3ex}}

\makeatletter
\newcommand*{\mint}[1]{%
  \mint@l{#1}{}%
}
\newcommand*{\mint@l}[2]{%
  \@ifnextchar\limits{%
    \mint@l{#1}%
  }{%
    \@ifnextchar\nolimits{%
      \mint@l{#1}%
    }{%
      \@ifnextchar\displaylimits{%
        \mint@l{#1}%
      }{%
        \mint@s{#2}{#1}%
      }%
    }%
  }%
}
\newcommand*{\mint@s}[2]{%
  \@ifnextchar_{%
    \mint@sub{#1}{#2}%
  }{%
    \@ifnextchar^{%
      \mint@sup{#1}{#2}%
    }{%
      \mint@{#1}{#2}{}{}%
    }%
  }%
}
\def\mint@sub#1#2_#3{%
  \@ifnextchar^{%
    \mint@sub@sup{#1}{#2}{#3}%
  }{%
    \mint@{#1}{#2}{#3}{}%
  }%
}
\def\mint@sup#1#2^#3{%
  \@ifnextchar_{%
    \mint@sup@sub{#1}{#2}{#3}%
  }{%
    \mint@{#1}{#2}{}{#3}%
  }%
}
\def\mint@sub@sup#1#2#3^#4{%
  \mint@{#1}{#2}{#3}{#4}%
}
\def\mint@sup@sub#1#2#3_#4{%
  \mint@{#1}{#2}{#4}{#3}%
}
\newcommand*{\mint@}[4]{%
  \mathop{}%
  \mkern-\thinmuskip
  \mathchoice{%
    \mint@@{#1}{#2}{#3}{#4}%
        \displaystyle\textstyle\scriptstyle
  }{%
    \mint@@{#1}{#2}{#3}{#4}%
        \textstyle\scriptstyle\scriptstyle
  }{%
    \mint@@{#1}{#2}{#3}{#4}%
        \scriptstyle\scriptscriptstyle\scriptscriptstyle
  }{%
    \mint@@{#1}{#2}{#3}{#4}%
        \scriptscriptstyle\scriptscriptstyle\scriptscriptstyle
  }%
  \mkern-\thinmuskip
  \int#1%
  \ifx\\#3\\\else_{#3}\fi
  \ifx\\#4\\\else^{#4}\fi  
}
\newcommand*{\mint@@}[7]{%
  \begingroup
    \sbox0{$#5\int\m@th$}%
    \sbox2{$#5\int_{}\m@th$}%
    \dimen2=\wd0 %
    \let\mint@limits=#1\relax
    \ifx\mint@limits\relax
      \sbox4{$#5\int_{\kern1sp}^{\kern1sp}\m@th$}%
      \ifdim\wd4>\wd2 %
        \let\mint@limits=\nolimits
      \else
        \let\mint@limits=\limits
      \fi
    \fi
    \ifx\mint@limits\displaylimits
      \ifx#5\displaystyle
        \let\mint@limits=\limits
      \fi
    \fi
    \ifx\mint@limits\limits
      \sbox0{$#7#3\m@th$}%
      \sbox2{$#7#4\m@th$}%
      \ifdim\wd0>\dimen2 %
        \dimen2=\wd0 %
      \fi
      \ifdim\wd2>\dimen2 %
        \dimen2=\wd2 %
      \fi
    \fi
    \rlap{%
      $#5%
        \vcenter{%
          \hbox to\dimen2{%
            \hss
            $#6{#2}\m@th$%
            \hss
          }%
        }%
      $%
    }%
  \endgroup
}

\allowdisplaybreaks

\begin{document}

\title[On sets with finite distributional fractional perimeter]{On sets with finite distributional fractional perimeter}

\author[G.~E.~Comi]{Giovanni E. Comi}
\address[G.~E.~Comi]{Dipartimento di Matematica, Università di Bologna, Piazza di Porta San Donato 5, 40126 Bologna (BO), Italy}
\email{giovannieugenio.comi@unibo.it}

\author[G.~Stefani]{Giorgio Stefani}
\address[G.~Stefani]{Scuola Internazionale Superiore di Studi Avanzati (SISSA), via Bonomea~265, 34136 Trieste (TS), Italy}
\email{gstefani@sissa.it {\normalfont or} giorgio.stefani.math@gmail.com}

\date{\today}

\begin{abstract}
We continue the study of the fine properties of sets having locally finite distributional fractional perimeter. 
We refine the characterization of their blow-ups and prove a Leibniz rule for the intersection of sets with locally finite distributional fractional perimeter with sets with finite fractional perimeter.
As a byproduct, we provide a description of non-local boundaries associated with the distributional fractional perimeter.
\end{abstract}

\keywords{Fractional gradient, fractional perimeter, Blow-up Theorem, Leibniz rule, non-local boundary}

\subjclass[2020]{Primary 26B30. Secondary 28A75}

\thanks{\textit{Acknowledgments}.
The authors thank the anonymous referees for several precious comments.
The authors are members of the Istituto Nazionale di Alta Matematica (INdAM), Gruppo Nazionale per l'Analisi Matematica, la Probabilità e le loro Applicazioni (GNAMPA).
The authors are partially supported by the INdAM--GNAMPA 2023 Project \textit{Problemi variazionali per funzionali e operatori non-locali}, codice CUP\_E53\-C22\-001\-930\-001.
The first-named author was partially supported by the INdAM--GNAMPA 2022 Project \textit{Alcuni problemi associati a funzionali integrali: riscoperta strumenti classici e nuovi sviluppi}, codice CUP\_E55\-F22\-000\-270\-001, and has received funding from the MIUR PRIN 2017 Project \textit{Gradient Flows, Optimal Transport and Metric Measure Structures}.
The second-named author was partially supported by the INdAM--GNAMPA 2022 Project \textit{Analisi geometrica in strutture subriemanniane}, codice CUP\_E55\-F22\-000\-270\-001, and has received funding from the European Research Council (ERC) under the European Union’s Horizon 2020 research and innovation program (grant agreement No.~945655).
}

\maketitle

\section{Introduction}

\subsection{The distributional fractional perimeter}

Given $\alpha\in(0,1)$, we let 
\begin{equation}
\label{eq:def_nabla_alpha}
\nabla^\alpha f(x)
=
\mu_{n,\alpha}
\int_{\R^n}\frac{(f(y)-f(x))(y-x)}{|y-x|^{n+\alpha+1}}\di y,
\quad
x\in\R^n,
\end{equation}
be the \emph{fractional $\alpha$-gradient} of $f\in\Lip_c(\R^n)$ and, analogously, 
\begin{equation}
\label{eq:def_div_alpha}
\div^\alpha\phi(x)
=
\mu_{n,\alpha}
\int_{\R^n}\frac{(\phi(y)-\phi(x))\cdot(y-x)}{|y-x|^{n+\alpha+1}}\di y,
\quad
x\in\R^n,
\end{equation}  
be the \emph{fractional $\alpha$-divergence} of $\phi\in\Lip_c(\R^n;\R^n)$, where 
\begin{equation*}
\mu_{n,\alpha}
=
2^\alpha\pi^{-\frac n2}\,\frac{\Gamma\left(\frac{n+\alpha+1}2\right)}{\Gamma\left(\frac{1-\alpha}2\right)}
>0
\end{equation*}
is a renormalization constant.
The operators~\eqref{eq:def_nabla_alpha} and~\eqref{eq:def_div_alpha} are \emph{dual}, in the sense that they satisfy the  fractional integration-by-parts formula
\begin{equation}
\label{eq:frac_ibp}
\int_{\R^n}f\,\div^\alpha\phi\di x
=
-\int_{\R^n}\phi\cdot\nabla^\alpha f\di x.
\end{equation}
For a more detailed account on the operators in~\eqref{eq:def_nabla_alpha} and~\eqref{eq:def_div_alpha} and on the formula~\eqref{eq:frac_ibp}, we refer the reader to~\cite{Silhavy20}.
In our previous papers~\cites{Comi-Stefani19,Comi-Stefani22-A,Comi-Stefani22-L,Comi-et-al21,Brue-et-al20,Comi-Stefani23-Fail,Comi-Stefani23-Frac}, starting from formula~\eqref{eq:frac_ibp}, we developed a new theory of distributional fractional Sobolev and $BV$ spaces. 

In the present note, we continue the study of fractional $BV$ functions.
Let us introduce the following definition (see~\cite{Comi-et-al21}*{Sec.~3.1} for example). 
Given $p\in[1,+\infty]$ and an open set $\Omega\subset\R^n$, we say that $f\in BV^{\alpha,p}_{\loc}(\Omega)$ if $f\in L^p(\R^n)$ and 
\begin{equation*}
\sup\set*{\int_{\R^n}f\,\div^\alpha\phi\di x : \phi\in C^\infty_c(\R^n;\R^n),\ \|\phi\|_{L^\infty(\R^n;\,\R^n)}\le 1,\ \supp\phi\subset A}
<+\infty
\end{equation*}
for any open set  $A\Subset\Omega$.
A simple application of Riesz's Representation Theorem (see~\cite{Comi-et-al21}*{Th.~3}) yields that $f\in BV^{\alpha,p}_\loc(\Omega)$ if and only if $f \in L^p(\R^n)$ and there is a vector-valued Radon measure $D^\alpha f\in\M_{\loc}(\Omega;\R^n)$, called \textit{fractional $\alpha$-variation measure} of $f$, such that 
\begin{equation}
\label{eq:ibp_BV_frac}
\int_{\R^n} f\div^\alpha\phi\di x
=
-\int_{\Omega}\phi \cdot \di D^\alpha f
\end{equation}
for all $\phi\in C^\infty_c(\R^n;\R^n)$ with $\supp\phi\subset\Omega$, with 
\begin{equation*}
|D^\alpha f|(A)
=
\sup\set*{\int_{\R^n}f\,\div^\alpha\phi\di x : \phi\in C^\infty_c(\R^n;\R^n),\ \|\phi\|_{L^\infty(\R^n;\,\R^n)}\le 1,\ \supp\phi\subset A}
\end{equation*}
for any open set $A\Subset\Omega$. 
If $|D^\alpha f|(\Omega)<+\infty$, then we write $f\in BV^{\alpha,p}(\Omega)$. 
We warn the reader that the subscript `$\loc$' in $BV^{\alpha,p}_\loc$ always refers to the local finiteness of the fractional variation measure only, as $BV^{\alpha,p}_\loc$ functions are in $L^p(\R^n)$ by default.

If $\chi_E\in BV^{\alpha,\infty}_\loc(\R^n)$, then the measure $|D^\alpha\chi_E|\in\M_\loc(\R^n)$ is called the \emph{distributional fractional} (\emph{Caccioppoli}) \emph{$\alpha$-perimeter} of $E\subset\R^n$ (see~\cite{Comi-Stefani19}*{Def.~4.1}). 
We recall that $W^{\alpha,1}(\R^n)\subset BV^{\alpha,1}(\R^n)$ with strict inclusion, see~\cite{Comi-Stefani19}*{Ths.~3.18 and~3.31}, so $|D^\alpha\chi_E|$ must not be confused with the \emph{fractional $\alpha$-perimeter} $P_\alpha(E;\,\cdot\,)$ relative to $W^{\alpha,1}$ sets, defined as
\begin{equation}
\label{eq:def_frac_per}
P_\alpha(E;\Omega)
=
\int_{(\R^n\times\R^n)\setminus(\Omega^c\times\Omega^c)}\frac{|\chi_E(x)-\chi_E(y)|}{|x-y|^{n+\alpha}}\di x\di y
\end{equation}
for $E,\Omega\subset\R^n$ (in particular, if $\Omega=\R^n$, then $P_\alpha(E)=P_\alpha(E;\R^n)=[\chi_E]_{W^{\alpha,1}(\R^n)}$).
In fact, if $P_\alpha(E;\Omega)<+\infty$, then $\chi_E\in BV^{\alpha,\infty}(\Omega)$ with $D^\alpha \chi_E = \nabla^\alpha \chi_E\,\Leb{n}$ and $\|\nabla^\alpha \chi_E\|_{L^1(\Omega;\,\R^n)}\le\mu_{n,\alpha}P_\alpha(E;\Omega)$, see~\cite{Comi-Stefani19}*{Prop.~4.8}, but currently we do not know if there exists $\chi_E\in BV^{\alpha,\infty}(\Omega)$ such that $P_\alpha(E;\Omega)=+\infty$.

Mimicking the classical theory (see~\cite{Comi-Stefani19}*{Sec.~4.5}), given $\chi_E\in BV^{\alpha,\infty}_\loc(\R^n)$, we say that $x\in\R^n$ belongs to the \emph{fractional reduced boundary} of $E$, and we write $x\in \redb^\alpha E$, if 
\begin{equation*}
x\in \supp |D^\alpha\chi_E|
\quad
\text{and}
\quad
\exists
\lim_{r\to0^+}
\frac{D^\alpha\chi_E(B_r(x))}{|D^\alpha\chi_E|(B_r(x))}
\in\mathbb S^{n-1}.
\end{equation*} 
Consequently, we let $\nu^\alpha_E\colon\redb^\alpha E\to\mathbb S^{n-1}$,
\begin{equation*}
\nu^\alpha_E(x)
=
\lim_{r\to0^+}
\frac{D^\alpha\chi_E(B_r(x))}{|D^\alpha\chi_E|(B_r(x))},
\quad
x\in\redb^\alpha E,
\end{equation*} 
be the \emph{(measure theoretic) inner unit fractional normal} of~$E$. Hence, \eqref{eq:ibp_BV_frac} implies that
\begin{equation}
\label{eq:frac_ibp_BV_sets}
\int_E\div^\alpha\phi\di x
=
-
\int_{\redb^\alpha E}\phi \cdot \nu^\alpha_E \, \di|D^\alpha \chi_E|
\end{equation} 
for all $\phi\in C^\infty_c(\R^n;\R^n)$.

\subsection{Leibniz rules}

A large part of our preceding works~\cites{Comi-Stefani19,Comi-Stefani22-A,Comi-Stefani22-L,Comi-et-al21,Brue-et-al20,Comi-Stefani23-Fail,Comi-Stefani23-Frac} is dedicated to the study of fractional Leibniz rules involving the operators~\eqref{eq:def_nabla_alpha} and~\eqref{eq:def_div_alpha}.
For $f,g\in\Lip_c(\R^n)$ and $\phi\in\Lip_c(\R^n;\R^n)$, we have 
\begin{equation}
\label{eq:nabla_prod}
\nabla^\alpha(fg)
=
f\,\nabla^\alpha g
+
g\,\nabla^\alpha f
+
\nabla^\alpha_\NL(f,g)
\end{equation}
and, analogously,
\begin{equation}
\label{eq:div_prod}
\div^\alpha_\NL(f\phi)
=
f\,\div^\alpha\phi
+
\phi\cdot\nabla^\alpha f
+
\div^\alpha_\NL(f,\phi),
\end{equation}
where 
\begin{equation}
\label{eq:def_nabla_NL}
\nabla^\alpha_\NL(f,g)(x)
=
\mu_{n,\alpha}\int_{\R^n}\frac{(f(y)-f(x))(g(y)-g(x))(y-x)}{|y-x|^{n+\alpha+1}}\di y,
\quad
x\in\R^n,
\end{equation}
and 
\begin{equation}
\label{eq:def_div_NL}
\div^\alpha_\NL(f,\phi)(x)
=
\mu_{n,\alpha}\int_{\R^n}\frac{(f(y)-f(x))(\phi(y)-\phi(x))\cdot(y-x)}{|y-x|^{n+\alpha+1}}\di y,
\quad
x\in\R^n,
\end{equation}
are the \emph{non-local fractional $\alpha$-gradient} and the \emph{non-local fractional $\alpha$-divergence}, respectively.
The fractional Leibniz rules~\eqref{eq:nabla_prod} and~\eqref{eq:div_prod}, as well as the  operators~\eqref{eq:def_nabla_NL} and~\eqref{eq:def_div_NL}, can be extended to less regular functions and vector fields in several ways, see~\cites{Comi-Stefani19,Comi-Stefani22-A,Comi-Stefani22-L,Comi-et-al21,Brue-et-al20,Comi-Stefani23-Fail,Comi-Stefani23-Frac}.

The first main aim of the present note is to achieve some new fractional Leibniz rules.
On the one side, we provide the following product rule for the intersection of $BV^{\alpha,\infty}_\loc$ sets with $W^{\alpha,1}$ sets,  generalizing~\cite{Comi-Stefani22-L}*{Th.~1.1}.
In fact, the locality property~\eqref{eq:schonberger} was inspired by an observation recently made in~\cite{Schonberger22}*{Rem.~3.4}. 
Here and in the following, $D^\alpha_\sing f$ denotes the singular part of the fractional variation measure $D^\alpha f$ of $f\in BV^{\alpha,p}_\loc(\R^n)$.
We also let 
\begin{equation*}
u^\star(x)
=
\lim_{r\to0^+}
\mint{-}_{B_r(x)} u(y)\di y,
\quad
x\in\R^n,
\end{equation*}
be the \emph{precise representative} of $u\in L^1_\loc(\R^n)$, whenever the limit exists in~$\R$.
We hence define the Borel set
\begin{equation*}
\rap_u
=
\set*{
x\in\R^n : u^\star(x)\ \text{exists in}\ \R
}.
\end{equation*}
Let us recall that $x\in\R^n$ is a \emph{Lebesgue point} of $u$ if $x\in\rap_u$ and 
\begin{equation*}
\lim_{r\to0^+}
\mint{-}_{B_r(x)} |u(y) - u^\star(x)|\di y = 0.
\end{equation*}
We note that, if $E\subset\R^n$ is a measurable set and $x\in\R^n$ is a Lebesgue point of $\chi_E$, then $\chi_E^\star(x) = \chi_{E^1}(x)$.
Here and below, for $t\in[0,1]$, we let
\begin{equation}
\label{eq:def_density_t}
E^{t}
=
\set*{x\in \R^{n} : \exists\lim_{r \to 0^+} 
\frac{|E \cap B_{r}(x)|}{|B_{r}(x)|} = t}.
\end{equation} 

\begin{theorem}[Intersection with $W^{\alpha,1}$ set]
\label{res:cap_W}
If $\chi_E\in BV^{\alpha,\infty}_\loc(\R^n)$ and $P_\alpha(F)<+\infty$, then $\chi_{E\cap F}\in BV^{\alpha,\infty}_\loc(\R^n)$, with
\begin{equation}
\label{eq:cap_BV_leibniz}
D^\alpha\chi_{E\cap F}
=
\chi_{F^1} D^\alpha\chi_E
+
\chi_E\nabla^\alpha\chi_F\,\Leb{n}
+
\nabla^\alpha_\NL(\chi_E,\chi_F)\,\Leb{n}
\quad
\text{in}\ 
\M_\loc(\R^n;\R^n),
\end{equation}
\begin{equation}
\label{eq:cap_BV_L1_part_est}
\max\set*{
\|\chi_E\nabla^\alpha\chi_F\|_{L^1(\R^n;\,\R^n)}
,
\|\nabla^\alpha_\NL(\chi_E,\chi_F)\|_{L^1(\R^n;\,\R^n)}
}
\le
\mu_{n,\alpha} P_\alpha(F) 
\end{equation}
and 
\begin{equation} \label{eq:zero_average_nabla_NL_E_F}
\int_{\R^n} \nabla^\alpha_\NL(\chi_E,\chi_F) \di x = 0.
\end{equation}
Consequently, we have 
\begin{equation} \label{eq:diff_grad_meas_1}
D^\alpha\chi_{E\cap F}
-
\chi_{F^1} D^\alpha\chi_E
\in
\M(\R^n;\R^n),
\end{equation} 
\begin{equation} \label{eq:diff_grad_meas_2}
|D^\alpha\chi_{E\cap F}
-
\chi_{F^1} D^\alpha\chi_E|
(\R^n)
\le
\mu_{n,\alpha}P_\alpha(F)
\end{equation} 
and
\begin{equation}
\label{eq:schonberger}
D^\alpha_\sing\chi_{E\cap F}
=
\chi_{F^1} D^\alpha_\sing\chi_E
\quad
\text{in}\
\M_\loc(\R^n;\R^n).
\end{equation}
In addition, if $F$ is also  bounded, then $\chi_{E\cap F}\in BV^{\alpha,\infty}(\R^n) \cap L^1(\R^n)$ and 
\begin{equation} \label{eq:GG_fract}
\int_{F^1} \, \di D^\alpha\chi_E = - \int_{E} \nabla^\alpha\chi_F\,\di x.
\end{equation}
\end{theorem}

We expect that \cref{res:cap_W} may be extended to $BV^{\alpha,\infty}_\loc$ (and even $BV^{\alpha,p}_\loc$) functions, but we do not pursue this direction here and leave it to future works.

On the other side, we generalize~\cite{Comi-Stefani22-L}*{Rem.~4.6}, see \cref{res:conditional_leibniz} below.
To state our result, we need to introduce some notation.
In analogy with \cite{Comi-Stefani22-L}*{Def.~4.5}, we exploit~\cite{Comi-Stefani22-L}*{Lem.~2.9} to define the measure version of the non-local fractional gradient~\eqref{eq:def_nabla_NL} (actually, we already used this object in~\cite{Comi-Stefani22-L}*{Th.~5.1} without providing its explicit definition).

\begin{definition}[Non-local fractional $\alpha$-gradient measure]
\label{def:weak_NL_grad}
Let  $p,q\in[1,+\infty]$ be such that $\frac1p+\frac1q\le1$.
Let $f\in L^p(\R^n)$ and $g\in L^q(\R^n)$. 
We say that $D^\alpha_{\NL}(f,g) \in \M_{\rm loc}(\R^n;\R^n)$ is a \emph{non-local fractional $\alpha$-gradient measure} of the pair $(f,g)$ if 
\begin{equation*}
\int_{\R^n}
f\,\div^\alpha_{\NL}(g,\varphi)
\di x
=
\int_{\R^n}
\varphi\cdot \di D^\alpha_{\NL}(f,g)
\quad 
\text{for all}\
\varphi\in C^\infty_c(\R^n;\R^n).
\end{equation*}
\end{definition}

Arguing as in~\cite{Comi-Stefani22-L}*{Sec.~4.4}, we can exploit~\cite{Comi-Stefani22-L}*{Cor.~2.7 and Lems.~2.9 and~2.10} to infer that \cref{def:weak_NL_grad} is well posed and that  $D^\alpha_{\NL} (f,g)$, if it exists, is unique and symmetric, and extends the operator~\eqref{eq:def_nabla_NL}.

Our second result on the Leibniz rule for $BV^{\alpha,\infty}_\loc$ functions can be stated as follows.

\begin{theorem}[Conditional Leibniz rule in $BV^{\alpha,\infty}_\loc$]
\label{res:conditional_leibniz}
If $f,g\in BV^{\alpha,\infty}_\loc(\R^n)$, then 
\begin{equation}
\label{eq:measure_NL_boh}
fg\in BV^{\alpha,\infty}_\loc(\R^n)
\iff
\exists \, D^\alpha_\NL(f,g)\in\M_\loc(\R^n;\R^n),
\end{equation}
and there exist $\bar f\in L^\infty(\R^n,|D^\alpha g|)$ and $\bar g\in L^\infty(\R^n,|D^\alpha f|)$, with 
\begin{equation}
\label{eq:repres_leibniz_BV_loc_bounds}
\|\bar f\|_{L^\infty(\R^n,|D^\alpha g|)}
\le 
\|f\|_{L^\infty(\R^n)}
\quad
\text{and}
\quad
\|\bar g\|_{L^\infty(\R^n,|D^\alpha f|)}
\le 
\|g\|_{L^\infty(\R^n)},
\end{equation} 
\begin{equation}
\label{eq:bar}
\bar f
=
f^\star\
\text{$|D^\alpha g|$-a.e.\ in $\rap_f$}
\quad
\text{and}
\quad
\bar g
=
g^\star\
\text{$|D^\alpha f|$-a.e.\ in $\rap_g$},
\end{equation} 
such that, provided that $fg\in BV^{\alpha,\infty}_\loc(\R^n)$,
\begin{equation}
\label{eq:conditional_leibniz}
D^\alpha(fg)
=
\bar f\,D^\alpha g
+
\bar g\,D^\alpha f
+
D^\alpha_\NL(f,g)
\quad
\text{in}\ \M_\loc(\R^n;\R^n).
\end{equation}
\end{theorem}

At the present moment, we do not know if the measure $D^\alpha_\NL(f,g)$ is well defined even in the case $f,g\in BV^{\alpha,\infty}(\R^n)\cap L^1(\R^n)$ with $|D^\alpha f|,|D^\alpha g|\ll\Leb{n}$, see~\cite{Comi-Stefani22-L}*{Rem.~4.6}.
However, in the simplest case $f=g=\chi_E\in BV^{\alpha,\infty}_\loc(\R^n)$, we have the following result.

\begin{corollary}[The measure $D^\alpha_\NL(\chi_E,\chi_E)$]
\label{res:chichi}
If $\chi_E\in BV^{\alpha,\infty}_\loc(\R^n)$, then
$D^\alpha_\NL(\chi_E, \chi_E) \in \M_\loc(\R^n; \R^n)$ is well defined and satisfies
\begin{equation} \label{eq:NL_grad_chi_E_repr}
D^\alpha_{\rm NL}(\chi_E, \chi_E) = (1 - 2\, \overline{\chi_E})\, D^\alpha \chi_E 
\quad 
\text{in}\ \M_{\loc}(\R^n; \R^n),
\end{equation}
where 
$\overline{\chi_E} \in L^\infty(\R^n, |D^\alpha \chi_E|)$, 
with 
\begin{equation}
\label{eq:bar_E}
0\le\overline{\chi_E}\le 1
\quad 
\text{and}
\quad 
\overline{\chi_E}=\chi_E^\star
\quad  
\text{$|D^\alpha\chi_E|$-a.e.\ in~$\rap_{\chi_E}$}.
\end{equation}
In addition, if $\chi_E \in BV^{\alpha, 1}(\R^n)$, then 
\begin{equation} \label{eq:zero_measure_R_n_NL}
D^\alpha_{\rm NL}(\chi_E, \chi_E) (\R^n) = \int_{\R^n} \overline{\chi_E} \di D^\alpha \chi_E = 0.
\end{equation}
\end{corollary}

In the limiting case $\alpha = 1$, the non-local gradient disappears, hence~\eqref{eq:NL_grad_chi_E_repr} reduces to $(1 - 2\, \overline{\chi_E} )\,D \chi_E = 0$, coherently with the fact that $\overline{\chi_E} = \chi_E^\star = \frac{1}{2}$ $|D \chi_E|$-a.e.\ in~$\R^n$.

\subsection{Analysis of blow-ups}

The main result of our first paper, see~\cite{Comi-Stefani19}*{Th.~5.8 and Prop.~5.9}, provides the following fractional counterpart of  De Giorgi's Blow-up Theorem for sets with locally finite perimeter (see~\cite{Maggi12}*{Part Two} for a detailed exposition).
Here and in the rest of the paper, for any measurable $E\subset\R^n$ and $x\in\R^n$, we let $\tang(E,x)$ be the set of all \emph{tangent sets to $E$ at $x$}, i.e., all limit points of $\set*{\frac{E-x}r : r>0}$ with respect to the convergence in $L^1_\loc(\R^n)$ as $r\to0^+$.

\begin{theorem}[Existence and rigidity of blow-ups]
\label{res:old_blowup}
If $\chi_E\in BV^{\alpha,\infty}_\loc(\R^n)$ and $x\in\redb^\alpha E$, then $\tang(E,x)\ne\emptyset$ and any $F\in\tang(E,x)$ is such that $\chi_F\in BV^{\alpha,\infty}_\loc(\R^n)$ with $\nu^\alpha_F(y)=\nu^\alpha_E(x)$ for $|D^\alpha\chi_F|$-a.e.\ $y\in\redb^\alpha F$.
\end{theorem} 

The second main aim of this note is to refine \cref{res:old_blowup}.
On the one hand, we provide the following convergence result, which was somehow implicit in the proof of \cite{Comi-Stefani19}*{Prop.~5.9}.

\begin{theorem}[Refined convergence] 
\label{res:convergence}
Let $\chi_E\in BV^{\alpha,\infty}_\loc(\R^n)$ and $x \in \redb^{\alpha} E$. 
If $F \in \tang(E, x)$ with  $\chi_{\frac{E - x}{r_{k}}} \to \chi_F$ in $L^{1}_\loc(\R^{n})$ as $r_{k} \to 0^+$, then, up to extracting a subsequence:
\begin{enumerate}[label=(\roman*),ref=\roman*,itemsep=.5ex,leftmargin=5ex]

\item 
\label{item:weak_conv_blow_up} $D^{\alpha} \chi_{\frac{E - x}{r_{k}}} \weakto D^{\alpha} \chi_F$
in $\M_{\loc}(\R^{n}; \R^{n})$
as $r_{k} \to 0^+$;

\item
\label{item:weak_conv_blow_up_tot_var} 
$|D^{\alpha} \chi_{\frac{E - x}{r_{k}}}| \weakto |D^{\alpha} \chi_F|$
in $\M_\loc(\R^{n})$
as $r_{k} \to 0^+$; 

\item
\label{item:conv_blow_up_tot_var_ball}
$D^{\alpha} \chi_{\frac{E - x}{r_{k}}}(B_{R}) \to D^{\alpha} \chi_F(B_{R})$
and 
$|D^{\alpha} \chi_{\frac{E - x}{r_{k}}}|(B_{R}) \to |D^{\alpha} \chi_F|(B_{R})$
for all $R>0$ such that $|D^\alpha \chi_F|(\partial B_R) = 0$ (in particular, for a.e. $R > 0$). 
\end{enumerate}
\end{theorem}

On the other hand, we provide the following characterization of blow-ups, which can be seen as a first step towards a fractional counterpart of De Giorgi's Structure Theorem for sets with locally finite perimeter, see \cite{Maggi12}*{Part 2} for instance.
Here and in the following, we let $\de^\alpha=D^\alpha_{\R}$ denote the fractional variation measure in dimension $n=1$ and we let $P$ be the standard De Giorgi's perimeter.

\begin{theorem}[Characterization of blow-ups]
\label{res:blowup}
Let $\chi_E\in BV^{\alpha,\infty}_{\loc}(\R^n)$, $x\in\redb^\alpha E$ and $\nu_E^\alpha(x)=\e_n$. 
If $F\in\tang(E,x)$, then $F=\R^{n-1}\times M$ with $M\subset\R$ such that:
\begin{enumerate}[label=(\roman*),ref=\roman*,itemsep=.5ex,leftmargin=5ex]

\item
\label{item:blowup_M_BV} 
$\chi_M\in BV^{\alpha,\infty}_\loc(\R)$ with $\de^\alpha\chi_M\ge0$;

\item
\label{item:blowup_M_meas_infty} 
$|M|,|M^c|\in\set*{0,+\infty}$;

\item
\label{item:blowup_M_esssup} 
if $|M| = + \infty$, then $\esssup M=+\infty$;

\item
\label{item:blowup_M_intervals} 
if $M\ne\emptyset,\R$ is such that $P(M)<+\infty$, then $M=(m,+\infty)$ for some $m\in\R$.
\end{enumerate}
\end{theorem}

\cref{res:blowup} shows a quite surprising similarity between the reduced boundary $\redb E$ for $\chi_E\in BV_{\loc}(\R^n)$ and the fractional reduced boundary $\redb^\alpha E$ for $\chi_E\in BV^{\alpha,\infty}_{\loc}(\R^n)$, going in the same direction of De Giorgi's Blow-up Theorem, see~\cite{Maggi12}*{Th.~15.5}.

A simple consequence of \cref{res:blowup} is that a cone cannot be a blow-up set on the fractional reduced boundary, unless it is a half-space, $\R^n$ or $\emptyset$.
Such a rigidity of blow-ups for example implies that no vertex of the square $E=[0,1]^2\subset\R^2$ belongs to $\redb^\alpha E$, for any $\alpha\in(0,1)$, in analogy with the case $\alpha=1$, see~\cite{Maggi12}*{Exam.~15.4}. 
Similarly, no vertex of the \textit{Koch snowflake} (see~\cite{Lombardini19}*{Sec.~3.3}) belongs to its fractional reduced boundary.

\cref{res:blowup} is a consequence of the following two results, which may be interesting on their own.
The first one characterizes $BV^{\alpha,\infty}_\loc$ functions with zero fractional derivative.

\begin{proposition}[Null derivative]
\label{res:null_derivativel}
Let $f\in BV^{\alpha,\infty}_\loc(\R^n)$ and let $i\in\set*{1,\dots,n}$.
Then, $D_i^\alpha f=0$ if and only if $D_i f=0$.
\end{proposition}

The second result allows to factorize the fractional variation measure of a $BV^{\alpha,\infty}_\loc$ function which does not depend on certain coordinates.  

\begin{proposition}[Splitting]
\label{res:splitting}
Let $f\in BV^{\alpha,\infty}_\loc(\R^n)$.
If $D_1f=0$, then there exists $g\in BV^{\alpha,\infty}_{\loc}(\R^{n-1})$ such that $f((t,x))=g(x)$ for a.e.\ $t\in\R$ and a.e.\ $x\in\R^{n-1}$ and 
\begin{equation*}
(D^\alpha_{\R^{n}})_if
=
\Leb{1}\otimes(D^\alpha_{\R^{n-1}})_i g
\quad
\text{in}\ 
\M_\loc(\R^n)\
\text{for all $i=2,\dots,n$.}
\end{equation*}
\end{proposition}

Propositions~\ref{res:null_derivativel} and~\ref{res:splitting} may hold for $BV^{\alpha,p}_\loc$ functions as well, but we leave this line of research for forthcoming works, being out of the scopes of the present note.

\subsection{Analysis of non-local boundaries}
The third and last main aim of this note is to exploit the above results on blow-ups and Leibniz rules for $BV^{\alpha,\infty}_\loc$ sets to infer some properties of  non-local boundaries linked with the distributional fractional perimeter.   

On the one side, given $\chi_E\in BV^{\alpha,\infty}_\loc(\R^n)$, we may decompose the fractional variation measure of~$E$ as  
$D^\alpha\chi_E=D^\alpha_\ass\chi_E+D^\alpha_\sing\chi_E$, where $|D^\alpha_\ass\chi_E|\ll\Leb{n}$ and $D^\alpha_\sing\chi_E\perp\Leb{n}$. 
In virtue of~\eqref{eq:schonberger} in \cref{res:cap_W}, $D^\alpha_\sing\chi_E$ has a local nature, in contrast with the non-local and thus `diffuse' behavior of the measure $D^\alpha_\ass\chi_E$.
The following result, which is a simple consequence of \cref{res:cap_W}, gives an idea of the size of the support of the measure $D^\alpha_\sing\chi_E$. 
Here and in the following, we let
\begin{equation*}
\partial^- E = \set*{ x \in \R^n : 0 < |E \cap B_r(x)| < |B_r(x)|\ \text{for  all}\ r > 0}.
\end{equation*}

\begin{theorem}[Support of $D^\alpha_\sing\chi_E$]
\label{res:supp_sing}
If $\chi_E\in BV^{\alpha,\infty}_\loc(\R^n)$, then
\begin{equation*}
|D^\alpha_\sing \chi_E|(F^1) = 0\
\text{whenever}\
\chi_F\in W^{\alpha,1}(\R^n)\
\text{with either}\
|E \cap F| = 0\
\text{or}\
|E^c \cap F| = 0.
\end{equation*}
In particular, $\supp|D^\alpha_\sing\chi_E|\subset\de^-E$.
\end{theorem}

\cref{res:supp_sing} has several analogies with classical results. Indeed, it is well-known that, if $\chi_E\in BV_\loc(\R^n)$, then
$\supp|D \chi_E| = \partial^- E$ 
(see~\cite{Maggi12}*{Prop.~12.19} for instance), while, if $E$ has locally finite fractional perimeter, then 
\begin{equation*}
\partial^- E 
= 
\set*{ x \in \R^n : P_{\alpha}^L(E; B_r(x)) > 0\ \text{for all}\ r > 0 },
\end{equation*}
where 
\begin{equation*}
P_{\alpha}^L(E; A) 
= 
\int_{E \cap A} \int_{A \setminus E} \frac{\di x\di y}{|x - y|^{n+\alpha}},
\quad
A\subset\R^n,
\end{equation*}
is the \textit{local part} of the fractional perimeter $P_\alpha(E;\,\cdot\,)$, see~\cite{Lombardini19}*{Lem.~3.1}. 
Furthermore, \cref{res:supp_sing} can be combined with~\cite{Comi-Stefani19}*{Cor.~5.4} to get the estimate
\begin{equation*}
|D^\alpha_\sing \chi_E| \le c_{n, \alpha} \Haus{n-\alpha} \mres (\redb^\alpha E \cap \partial^- E).
\end{equation*}
In particular, if $\Haus{n-\alpha}(\redb^\alpha E \cap \partial^- E) = 0$, then $|D^\alpha \chi_E| \ll \Leb{n}$.
However, we warn the reader that $\partial^- E$ may have positive Lebesgue measure, even for a set with finite perimeter (see~\cite{Maggi12}*{Exam.~12.25} for instance).

On the other side, in view of \cref{res:old_blowup,res:blowup}, we wish to study the set of points $x \in \redb^{\alpha} E$ for which fractional blow-ups are non-trivial, i.e., such that $\tang(E,x)\ne\set*{\emptyset}$ and $\tang(E,x)\ne\set*{\R^n}$.
As well-known, for any $E\subset\R^n$ measurable and $x\in\R^n$, we have 
\begin{equation}
\label{eq:tantan}
\tang(E,x)=\set*{\emptyset}\iff x\in E^{0},
\quad
\tang(E,x)=\set*{\R^n}\iff x\in E^{1},
\end{equation}
where $E^{0}$ and $E^{1}$ are as in~\eqref{eq:def_density_t} for $t=0,1$.
In order to avoid such cases, we consider
\begin{equation*}
\de^*E=\R^n\setminus(E^{0}\cup E^{1}).
\end{equation*}
On the other hand, it is also worth recalling that blow-up limits for $\chi_E\in BV_\loc(\R^n)$ may be not half-spaces and yet not trivial only when considering points $x\in\de^*E\setminus\redb E$.   
Moreover, at such points, $\tang(E,x)$ can be very wild.
Indeed, as proved in~\cite{Leonardi00}*{Prop.~2.7}, for $n\ge2$ there exists a measurable set $E\subset\R^n$ with $P(E)<\infty$ and such that 
\begin{equation*}
\tang(E,0)
\supset
\set*{F\subset\R^n\ \text{measurable} : P(F)<+\infty}.
\end{equation*}
Note that, for such set $E$, it holds $\chi_E\in BV^{\alpha,\infty}_\loc(\R^n)$ for any $\alpha\in(0,1)$ due to \cite{Comi-Stefani19}*{Prop.~4.8}, so that $0\in\de^*E\setminus\redb^\alpha E$ by  \cref{res:blowup}.

The equivalences in~\eqref{eq:tantan} and the example above motivate the following definition.

\begin{definition}[Effective fractional reduced boundary] 
\label{def:eff}
Given $\chi_E\in BV^{\alpha,\infty}_\loc(\R^n)$, we let 
$\redb^{\alpha}_\eff E 
= 
\redb^{\alpha} E \cap \de^* E$ 
be the \emph{effective fractional reduced boundary} of~$E$.
\end{definition}

We recall that  $\dimension_\Haus{}(\redb^\alpha E)\ge n-\alpha$, see~\cite{Comi-Stefani19}*{Prop.~5.5}, possibly with strict inequality.
In fact, from~\eqref{eq:frac_ibp_BV_sets}, we immediately deduce that, if $|D^\alpha \chi_E| \ll \Leb{n}$ and $|D^{\alpha} \chi_{E}|(\Omega) > 0$, then $|\Omega \cap \redb^{\alpha} E| > 0$ as well. 
Indeed, in this case, we have
\begin{equation*}
\int_E \div^\alpha \varphi \di x = \int_{\Omega \cap \redb^{\alpha} E} \varphi \cdot \di D^\alpha_\ass\chi_E \ \text{ for all } \varphi \in C^{\infty}_c(\Omega; \R^n),
\end{equation*}
so that $|\Omega \cap \redb^{\alpha} E| = 0$ implies $|D^{\alpha} \chi_{E}|(\Omega) = 0$. 
In particular, this applies to the case $P_\alpha(E, \Omega) < + \infty$, see~\cite{Comi-Stefani19}*{Rem.~4.9}.
The following result refines such statement for the effective fractional reduced boundary. 
Here and below, given $\nu \in\mathbb S^{n - 1}$ and $x_0 \in \R^n$,
\begin{equation*}
H_{\nu}^{+}(x_0) 
= \set*{ y \in \R^{n} : (y - x_0) \cdot \nu \ge 0 },
\quad
H_{\nu}(x_0) = \set*{ y \in \R^{n} : (y-x_0) \cdot \nu = 0 },
\end{equation*}
with the shorthands $H_\nu^+=H_\nu^+(0)$ and $H_\nu=H_\nu(0)$.

\begin{theorem}[Properties of $\redb^\alpha_\eff E$] \label{res:eff_bello}
\quad
\begin{enumerate}[label=(\roman*),ref=\roman*,itemsep=.5ex,leftmargin=5ex]

\item 
\label{item:eff_BV_frac}
If $\chi_E\in BV^{\alpha,\infty}_\loc(\R^n)$, $x\in\redb^\alpha_\eff E$ and $\tang(E,x)=\set*{F}$, then $F=H_{\nu_{E}^\alpha(x)}^+$, and
\begin{equation*}
\Sigma_E
=
\set*{x\in\redb^\alpha_\eff E : \text{$E$ admits a unique blow-up at $x$}}
\end{equation*}
can be covered by countably many $(n-1)$-dimensional Lipschitz graphs.

\item
\label{item:eff_W}
If $\chi_{E} \in W^{\alpha, 1}_\loc(\R^{n})$, then 
$\Haus{n-\alpha}(\redb^{\alpha}_\eff E)=0$. 

\item
\label{item:eff_BV}
If $\chi_E\in BV_\loc(\R^n)$, then $\redb E \subset \redb^{\alpha}_\eff E$,
$\Haus{n - 1}(\redb^{\alpha}_\eff E \setminus \redb E) = 0$ 
and 
$\nu_{E}^{\alpha} = \nu_{E}$ on~$\redb E$.

\end{enumerate}
\end{theorem}

The proof of \cref{res:eff_bello} combines \cref{def:eff} and the properties of the fractional reduced boundary established so far with several known results concerning blow-ups.
Point~\eqref{item:eff_BV_frac} is a consequence of \cref{res:blowup}, \cite{Leonardi00}*{Prop.~2.1} and~\cite{DelNin21}*{Th.~1.2}.
Point~\eqref{item:eff_W} exploits~\cite{Ponce-Spector20}*{Prop.~3.1}.
Finally, point~\eqref{item:eff_BV} follows from well-known properties of sets with locally finite perimeter as soon as the inclusion $\redb E\subset\redb^\alpha E$ is proved. 

In the proof of \cref{res:eff_bello} we exploit the following result, which can be easily deduced from~\cite{Comi-Stefani23-Fail}*{Prop.~1.8} and in fact provides a notable example for point~\eqref{item:eff_BV} in \cref{res:eff_bello}.

\begin{proposition}[Half-space] \label{res:halfspace}
If $x_0 \in \R^n$ and $\nu \in\mathbb S^{n-1}$, then
\begin{equation*} \nabla^{\alpha} \chi_{H_{\nu}^{+}(x_0)}(x) = \frac{\mu_{1, \alpha}}{\alpha} \frac{\nu}{|(x-x_0) \cdot \nu|^{\alpha}} 
\quad  
\text{for all}\ x \notin H_{\nu}(x_0),
\end{equation*}
with $\redb^{\alpha} H^{+}_{\nu}(x_0) = \R^{n}$,
$\redb^{\alpha}_\eff H^{+}_{\nu}(x_0) = H_{\nu}(x_0)$ 
and 
$\nu_{H^{+}_{\nu}(x_0)}^{\alpha}=\nu$ on~$\R^n$.
\end{proposition}

\cref{res:halfspace}, combined with \cref{res:convergence,res:eff_bello}\eqref{item:eff_BV_frac}, implies the following result.

\begin{corollary}[Refined convergence on $\Sigma_E$]
Let $\chi_E\in BV^{\alpha,\infty}_{\loc}(\R^n)$ and $x\in\Sigma_E$.
If $\chi_{\frac{E-x}{r_k}}\to\chi_{H^+_{\nu^\alpha_E(x)}}$ in $L^1_{\loc}(\R^n)$ as $r_k\to0^+$, then, up to extracting a subsequence, as $r_{k} \to 0^+$,
\begin{align*}
D^{\alpha} \chi_{\frac{E - x}{r_{k}}} &\weakto \frac{\mu_{1, \alpha}}{\alpha} \frac{\nu^\alpha_E(x)}{|(\,\cdot\,) \cdot \nu^\alpha_E(x)|^{\alpha}} \,\Leb{n}
\quad
\text{in $\M_{\loc}(\R^{n}; \R^{n})$},
\\[1ex]
|D^{\alpha} \chi_{\frac{E - x}{r_{k}}}| &\weakto \frac{\mu_{1, \alpha}}{\alpha} \frac{1}{|(\,\cdot\,) \cdot \nu^\alpha_E(x)|^{\alpha}}\,\Leb{n}
\quad
\text{in $\M_\loc(\R^{n})$}.
\end{align*}
\end{corollary} 

We conclude our paper with the following result, which is a simple consequence of~\cite{Comi-Stefani19}*{Th.~3.18} and provides another important example for point~\eqref{item:eff_BV} in \cref{res:eff_bello}. 

\begin{proposition}[Ball]
\label{res:ball}
If $x_0\in\R^n$ and $r>0$, then
\begin{equation*}
\nabla^{\alpha} \chi_{B_{r}(x_0)}(x) = - \frac{\mu_{n, \alpha}}{n + \alpha - 1} \int_{\de B_{r}(x_0)} \frac{y}{|x-y|^{n + \alpha - 1}} \di \Haus{n - 1}(y)
\quad
\text{for all}\ x\notin\de B_r(x_0),
\end{equation*}
with $\redb^\alpha B_r(x_0)=\R^n\setminus\set*{x_0}$,
$\redb^\alpha_\eff B_r(x_0)=\de B_r(x_0)$
and 
$\nu^\alpha_{B_r(x_0)}(x)=-\tfrac{x-x_0}{|x-x_0|}$ for all $x\ne x_0$.
\end{proposition}

\section{Proofs of  the results}

The rest of the paper is devoted to the proofs of our results.
Throughout this section, we let $(\rho_\eps)\subset C^\infty_c(\R^n)$ be a family of standard mollifiers, that is, we let $\rho_\eps(x)=\eps^{-n}\rho(\frac x\eps)$ for $x\in\R^n$ and $\eps>0$, where 
\begin{equation}
\label{eq:kernel}
\rho \in C^\infty_c(\R^n),
\quad
\rho\ge0,
\quad
\rho\ \text{is radial},
\quad
\supp\rho\subset B_1
\quad
\text{and}
\quad
\int_{B_1} \rho \di x = 1.
\end{equation}
In addition, for $x \in \R^n$ and $i \in \{1, \dots, n\}$, we let $\hat{x}_i\in\R^{n-1}$ be defined as 
\begin{equation*}
\hat{x}_i 
=
\begin{cases} 
(x_2, \dots, x_n) 
& \text{if}\ 
i = 1, 
\\
(x_1, \dots, x_{i-1}, x_{i+1}, \dots , x_n) 
& 
\text{if}\ 
i \in \{2, \dots, n-1\}, 
\\
(x_1, \dots, x_{n-1}) 
& 
\text{if}\ i = n.
\end{cases}
\end{equation*}

In some of the proofs below, we will invoke the following result (only for $p=+\infty$).

\begin{lemma}[Smoothing]
\label{res:smoothing}
Let $p\in[1,+\infty]$.
If $f\in BV^{\alpha,p}_\loc(\R^n)$ and $\rho\in C_c^\infty(\R^n)$, then $\rho*f\in BV^{\alpha,p}_\loc(\R^n)\cap C^\infty(\R^n)$ with $\nabla^\alpha(\rho*f)=\rho*D^\alpha f$ in $L^1_\loc(\R^n;\R^n)$.
\end{lemma}

\begin{proof}
Clearly, $\rho*f\in W^{1,p}(\R^n)$, hence $\nabla^\alpha(\rho*f)\in L^p(\R^n;\R^n)$ by~\cite{Comi-Stefani22-A}*{Prop.~3.3}. 
Moreover, $\rho*D^\alpha f\in L^1_\loc(\R^n;\R^n)$. 
Given $\phi\in C^\infty_c(\R^n;\R^n)$, we have $\rho*\phi\in C^\infty_c(\R^n;\R^n)$ and thus 
\begin{equation*}
\int_{\R^n}(\rho*f)\,\div^\alpha\phi\di x
=
\int_{\R^n}f\,(\rho*\div^\alpha\phi)\di x
=
\int_{\R^n}f\,\div^\alpha(\rho*\phi)\di x
=
-\int_{\R^n}(\rho*\phi) \, \di D^\alpha f,
\end{equation*}
thanks to \cite{Comi-Stefani19}*{Lem.~3.5}, readily yielding $\nabla^\alpha(\rho*f)=\rho*D^\alpha f$ in $L^1_{\loc}(\R^n;\R^n)$.
\end{proof}

\subsection{Proof of \texorpdfstring{\cref{res:cap_W}}{Theorem 1.1}}

Given $\phi\in C^\infty_c(\R^n;\R^n)$, since $P_\alpha(F)<+\infty$, arguing as in the proof of \cite{Comi-Stefani22-L}*{Lem.~3.2}, we can write 
\begin{equation}
\label{eq:leibniz_BV_easy}
\div^\alpha(\phi\chi_F)
=
\chi_F\,\div^\alpha\phi
+\phi\cdot\nabla^\alpha\chi_F
+
\div^\alpha_\NL(\chi_F,\phi)
\quad
\text{in}\ L^1(\R^n).
\end{equation}
Let us observe that 
\begin{equation*}
\bigg|
\int_{\R^n}\chi_E\,\phi\cdot\nabla^\alpha\chi_F\di x
\,\bigg|
\le 
\|\phi\|_{L^\infty(\R^n;\,\R^n)}
\,
\|\nabla^\alpha\chi_F\|_{L^1(\R^n;\,\R^n)}.
\end{equation*}
Moreover, by~\cite{Comi-Stefani22-L}*{Lem.~2.9}, we have  
\begin{equation}
\label{eq:NL_ibp}
\int_{\R^n}\chi_E\,\div^\alpha_\NL(\chi_F,\phi)\di x
=
\int_{\R^n}
\phi\cdot\nabla^\alpha_\NL(\chi_E,\chi_F)\di x,
\end{equation}
so that
\begin{equation*}
\bigg|
\int_{\R^n}\chi_E\,\div^\alpha_\NL(\chi_F,\phi)\di x
\,\bigg|
\le 
\|\phi\|_{L^\infty(\R^n;\,\R^n)}
\,
\|\nabla^\alpha_\NL(\chi_E,\chi_F)\|_{L^1(\R^n;\,\R^n)}.
\end{equation*} 
Recalling the definitions in~\eqref{eq:def_frac_per} and~\eqref{eq:def_nabla_NL}, as in~\cite{Comi-Stefani22-L}*{Cors.~2.3 and~2.7} (since $|\chi_E(x)-\chi_E(y)|\le1$ for $x,y\in\R^n$) we can estimate
\begin{equation*}
\max\set*{\|\nabla^\alpha\chi_F\|_{L^1(\R^n;\,\R^n)},\|\nabla^\alpha_\NL(\chi_E,\chi_F)\|_{L^1(\R^n;\,\R^n)}}
\le 
\mu_{n,\alpha}\,P_\alpha(F),
\end{equation*}
in particular proving~\eqref{eq:cap_BV_L1_part_est}.
A straightforward application of~\cite{Comi-Stefani22-L}*{Lem.~2.9} (by taking $p,r=+\infty$, $q=1$, $f=\chi_E$, $g=\chi_F$ and $\phi\equiv1$ in that result, and observing that $b^\alpha_{1,1}(\R^n)$ coincides with the space of $L^1_\loc(\R^n)$ functions with finite $W^{\alpha,1}$-seminorm, see~\cite{Comi-Stefani22-L}*{Sec.~2.1}) gives~\eqref{eq:zero_average_nabla_NL_E_F}.
Now, integrating~\eqref{eq:leibniz_BV_easy} and exploiting~\eqref{eq:NL_ibp}, we get
\begin{equation*}
\int_{\R^n}
\chi_{E\cap F}\,\div^\alpha\phi
\di x
=
\int_E
\div^\alpha(\phi\chi_F)
\di x
-
\int_E
\phi\cdot\nabla^\alpha\chi_F
\di x
-
\int_{\R^n}
\phi\cdot
\nabla^\alpha_\NL(\chi_E,\chi_F)
\di x
\end{equation*}
Now let $\psi_\eps=\rho_\eps*(\phi\chi_F)$ for all $\eps>0$.
Since $\psi\in C^\infty_c(\R^n)$, we can compute
\begin{equation*}
\int_E\div^\alpha\psi_\eps\di x
=
-
\int_{\R^n}\psi_\eps \cdot \di D^\alpha\chi_E
\end{equation*} 
for all $\eps>0$. 
Since $\psi_\eps\to\phi\chi_F$ in $W^{\alpha,1}(\R^n; \R^n)$ as~$\eps\to0^+$ (see \cite{Comi-Stefani19}*{App.~A}), we have $\div^\alpha\psi_\eps\to\div^\alpha(\phi\chi_F)$ in~$L^1(\R^n)$ as~$\eps\to0^+$ and thus
\begin{equation*}
\lim_{\eps\to0^+}
\int_E\div^\alpha\psi_\eps\di x
=
\int_E\div^\alpha(\phi\chi_F)\di x.
\end{equation*} 
Since $\psi_\eps\to\phi\chi_{F^1}$ $\Haus{n-\alpha}$-a.e.\ in~$\R^n$ as~$\eps\to0^+$ by~\cite{Ponce-Spector20}*{Prop.~3.1} and the fact that $|D^\alpha\chi_E|\ll\Haus{n-\alpha}\mres\redb^\alpha E$ due to~\cite{Comi-Stefani19}*{Cor.~5.4}, we also get that 
\begin{equation*}
\lim_{\eps\to0^+}
\int_{\R^n}\psi_\eps \cdot \di D^\alpha\chi_E
=
\int_{\R^n}
\chi_{F^1} \phi \cdot \di D^\alpha\chi_E.
\end{equation*}
Hence, we get
\begin{equation*}
\int_{\R^n}
\chi_{E\cap F}\,\div^\alpha\phi
\di x
=
-
\int_{\R^n}
\chi_{F^1} \phi \cdot \di D^\alpha\chi_E
-
\int_E
\phi\cdot\nabla^\alpha\chi_F
\di x
-
\int_{\R^n}
\phi\cdot
\nabla^\alpha_\NL(\chi_E,\chi_F)
\di x
\end{equation*}
whenever $\phi\in C^\infty_c(\R^n;\R^n)$, which easily implies~\eqref{eq:cap_BV_leibniz} via a standard approximation argument. The validity of~\eqref{eq:diff_grad_meas_1}, \eqref{eq:diff_grad_meas_2} and~\eqref{eq:schonberger} is an easy consequence of~\eqref{eq:cap_BV_leibniz}. 
Finally, if $F$ is bounded, then $\chi_{E \cap F} \in BV^{\alpha,1}(\R^n)$, so that \cite{Comi-Stefani22-L}*{Lem.~2.5} implies $D^\alpha \chi_{E \cap F}(\R^n) = 0$, and so~\eqref{eq:GG_fract} follows by integrating~\eqref{eq:cap_BV_leibniz} over $\R^n$ and exploiting~\eqref{eq:zero_average_nabla_NL_E_F}.
\qed

\subsection{Proof of \texorpdfstring{\cref{res:conditional_leibniz}}{Theorem 1.3}}

Let $\phi\in C^\infty_c(\R^n;\R^n)$ and set $g_\eps=\rho_\eps*g$ for all $\eps>0$.
Since $g_\eps\in\Lip_b(\R^n)$,  by~\cite{Comi-Stefani22-A}*{Lem.~2.4} we can write 
\begin{equation*}
\div^\alpha(g_\eps\phi)
=
g_\eps\,\div^\alpha\phi
+
\phi\cdot\nabla^\alpha g_\eps
+
\div^\alpha_\NL(g_\eps,\phi)
\quad
\text{in}\ L^1(\R^n)\cap L^\infty(\R^n).
\end{equation*}
Hence, in virtue of~\cite{Comi-Stefani22-L}*{Lem.~2.10} and \cref{res:smoothing}, we have 
\begin{equation}
\label{eq:cardio}
\begin{split}
\int_{\R^n}fg_\eps \div^\alpha\phi\di x
&=
\int_{\R^n}f\,\div^\alpha(g_\eps\phi)\di x
-
\int_{\R^n}f\phi\cdot\nabla^\alpha g_\eps
\di x
-
\int_{\R^n}f\,\div^\alpha_\NL(g_\eps,\phi)\di x
\\
&=
-
\int_{\R^n}g_\eps\phi \cdot \di D^\alpha f
-
\int_{\R^n}\rho_\eps*(f\phi) \cdot \di D^\alpha g
-
\int_{\R^n}g_\eps \div^\alpha_\NL(f,\phi)\di x.
\end{split}
\end{equation}
Since $\div^\alpha\phi,\div^\alpha_\NL(f,\phi)\in L^1(\R^n)$, by the Dominated Convergence Theorem we have
\begin{equation}
\label{eq:rici}
\lim_{\eps\to0^+}
\int_{\R^n}fg_\eps\,\div^\alpha\phi\di x
=
\int_{\R^n}fg\,\div^\alpha\phi\di x,
\end{equation}
\begin{equation}
\label{eq:clag}
\lim_{\eps\to0^+}
\int_{\R^n}g_\eps\,\div^\alpha_\NL(f,\phi)\di x
=
\int_{\R^n}g\,\div^\alpha_\NL(f,\phi)\di x.
\end{equation}
Setting $f_\eps=\rho_\eps*f$ for all $\eps>0$ and letting $R>0$ be such that $\supp\phi\subset B_R$, we also have
\begin{equation*}
\begin{split}
\bigg|
\int_{\R^n}\rho_\eps*(f\phi) \cdot \di D^\alpha g
&-
\int_{\R^n} f_\eps\phi \cdot \di D^\alpha g
\,\bigg|
\le 
\int_{B_{R+1}}|\rho_\eps*(f\phi)
-
f_\eps\phi|\di |D^\alpha g|
\\
&=
\int_{B_{R+1}}\int_{\R^n}\rho_\eps(x-y)|f(y)||\phi(y)-\phi(x)|\di y\di |D^\alpha g|(x)
\\
&\le 
\|f\|_{L^\infty(\R^n)}
\int_{B_{R+1}}\int_{\R^n}\rho_\eps(x-y)|\phi(y)-\phi(x)|\di y\di |D^\alpha g|(x)
\end{split}
\end{equation*}
for all $\eps\in(0,1)$.
Since $\phi$ is uniformly continuous on $\R^n$, we thus get that 
\begin{equation}
\label{eq:gio}
\lim_{\eps\to0^+}
\int_{\R^n}\rho_\eps*(f\phi) \cdot \di D^\alpha g
=
\lim_{\eps\to0^+}
\int_{\R^n} f_\eps\phi \cdot \di D^\alpha g.
\end{equation}
Now, since $|f_\eps(x)| \le \|f\|_{L^\infty(\R^n)}$ for each $x \in \R^n$ and $\eps > 0$, the family $(f_\eps)_{\eps > 0}$ is uniformly bounded in $L^\infty(\R^n,|D^\alpha g|)$. 
Thus, by Banach--Alaoglu Theorem, we can find a sequence $(f_{\eps_k})_{k \in \N}$ and $\bar f\in L^\infty(\R^n,|D^\alpha g|)$ such that $f_{\eps_k} \weakstarto \bar f$ in $L^\infty(\R^n,|D^\alpha g|)$ as $k\to+\infty$. 
Similarly (up to subsequences, which we do not relabel), we can also find a sequence $(g_{\eps_k})_{k \in \N}$ and $\bar g \in L^\infty(\R^n,|D^\alpha f|)$ such that $g_{\eps_k} \weakstarto \bar g$ in $L^\infty(\R^n,|D^\alpha f|)$ as $k\to+\infty$. 
In particular, \eqref{eq:repres_leibniz_BV_loc_bounds} follows by the lower semicontinuity of the $L^\infty$ norm with respect to the weak$^*$ convergence. 
Therefore,  passing to the limit as $k\to+\infty$ in~\eqref{eq:cardio} along the sequence $(\eps_k)_{k\in\N}$ and recalling~\eqref{eq:rici}, \eqref{eq:clag} and~\eqref{eq:gio}, we get 
\begin{equation*}
\int_{\R^n}fg\,\div^\alpha\phi\di x
=
-
\int_{\R^n}\bar g \phi\cdot\di D^\alpha f
-
\int_{\R^n}\bar f \phi\cdot\di D^\alpha g
-
\int_{\R^n}g\,\div^\alpha_\NL(f,\phi)\di x
\end{equation*}
whenever $\phi\in C^\infty_c(\R^n;\R^n)$, readily yielding~\eqref{eq:measure_NL_boh} thanks to \cref{def:weak_NL_grad}, and therefore~\eqref{eq:conditional_leibniz} as long as $fg \in BV^{\alpha, \infty}_{\loc}(\R^n)$. 
To conclude, we thus just need to prove~\eqref{eq:bar}.
To this aim, we observe that, by \cref{lem:pointwise_limit_convolution} below, $f_{\eps_k}(x) \to f^\star(x)$ as $k\to+\infty$ for all $x \in \rap_f$.
Hence, by the Dominated Convergence Theorem, we get
\begin{equation*} 
\int_{\rap_f} \bar f \psi \cdot \di D^\alpha g 
= 
\lim_{k \to + \infty} \int_{\rap_f} f_{\eps_k} \psi \cdot \di D^\alpha g
=
\int_{\rap_f} f^\star \psi \cdot \di D^\alpha g
\end{equation*}
for any $\psi \in C_c(\R^n)$, proving the first half of~\eqref{eq:bar}.
The second half of~\eqref{eq:bar} is similar.\qed

\begin{lemma} \label{lem:pointwise_limit_convolution}
If $u\in L^1_{\loc}(\R^n)$ and  $x\in\rap_u$, then
\begin{equation*}
u^\star(x)
=
\lim_{\eps\to0^+}(\rho_\eps*u)(x).
\end{equation*}
\end{lemma}

\begin{proof}
In view of~\eqref{eq:kernel}, we can write $\rho(x)=\eta(|x|)$ for all $x\in\R^n$, where $\eta\in C^\infty_c([0,+\infty))$ is such that $\eta\ge0$ and $\supp\eta\subset[0,1)$. In particular, $\eta$ is absolutely continuous on $[0, 1]$. Hence, if $\eta$ is strictly decreasing on its support, we exploit Cavalieri's formula to get
\begin{equation*}
\begin{split}
(\rho_{\eps}*u)(x) 
&= 
\eps^{-n}
\int_{0}^{+\infty}\int_{\set*{y \in B_\eps\,:\,\rho(y / \eps) > t }} u(x - y)\di y\di t 
\\
&= 
-\eps^{-n-1}
\int_{0}^{\eps} \eta' \left(\tfrac{s}{\eps}\right)\int_{\set*{y\in B_\eps\,:\,\eta(|y|/ \eps) > \eta(s / \eps) }} u(x - y) \di y \di s 
\\ 
&= 
-\eps^{-n-1}
\int_{0}^{\eps}\eta'\left(\tfrac{s}{\eps}\right) \int_{B_s} u(x-y)\di y \di s
\\ 
&= 
- \int_{0}^{1}\eta'(r)\,\omega_{n} r^{n} \mint{-}_{B_{r \eps}(x)} u(y) \di y \di r. 
\end{split}
\end{equation*}
If $x\in\rap_u$, then we can apply the Dominated Convergence Theorem to get 
\begin{equation*}
\lim_{\eps\to0^+}
(\rho_\eps*u)(x)
=
- \int_{0}^{1}\eta'(r)\,\omega_{n} r^{n}\left(\lim_{\eps\to0^+}\mint{-}_{B_{r \eps}(x)} u(y) \di y \right) \di r
=
u^\star(x),
\end{equation*}
as we easily recognize that 
\begin{equation*}
-\int_{0}^{1} \eta'(r)\,\omega_{n}r^{n} \di r
=
n \omega_{n} \int_{0}^{1} \eta(r) r^{n - 1} \di r
=
\int_{B_1} \rho(x) \di x
=1.
\end{equation*}
In the general case, since $\eta$ is absolutely continuous on $[0,1]$, we can write $\eta=\eta_1-\eta_2$ on $[0,1]$, where $\eta_1, \eta_2 \colon [0, 1] \to [0, + \infty)$ are strictly decreasing  absolutely continuous functions such that $\eta_1(1)=\eta_2(1)$.
The conclusion hence follows by performing analogous computations involving Cavalieri's formula on $\eta_1$ and $\eta_2$ separately, and then by exploiting the linearity of the derivative.
\end{proof}

\subsection{Proof of \texorpdfstring{\cref{res:chichi}}{Corollary 1.4}}

The validity of~\eqref{eq:NL_grad_chi_E_repr} immediately follows from~\eqref{eq:conditional_leibniz} in \cref{res:conditional_leibniz}, since $\chi_E\,\chi_E=\chi_E\in BV^{\alpha,\infty}_\loc(\R^n)$.
In particular, $\overline{\chi_E} \in L^\infty(\R^n, |D^\alpha \chi_E|)$ given by \cref{res:conditional_leibniz} is uniquely determined by~\eqref{eq:NL_grad_chi_E_repr}. 
The validity of~\eqref{eq:bar_E} follows from~\eqref{eq:bar} and by the construction in the proof of \cref{res:conditional_leibniz}. 
Finally, if $\chi_E \in BV^{\alpha,1}(\R^n)$, then \eqref{eq:NL_grad_chi_E_repr} easily implies that $D^\alpha_\NL(\chi_E, \chi_E) \in \M(\R^n; \R^n)$. We let $\eta \in C^{\infty}_c(B_2)$ be such that $\eta \equiv 1$ on $B_1$ and set $\eta_k(x) = \eta \left (\frac{x}{k} \right )$ for $k\in\N$ and $x\in\R^n$. By \cref{def:weak_NL_grad} with $\varphi_k = \eta_k {\rm e}_j$ for $j \in \{1, \dots, n\}$ and~\cite{Comi-Stefani22-L}*{Cor.~2.7}, we get
\begin{align*}
\left |\int_{\R^n} \eta_k\, {\rm e}_j \cdot \di D^\alpha_{\NL}(\chi_E, \chi_E)\right| 
& = 
\left | \int_{\R^n}  \chi_E\, {\rm e}_j \cdot \nabla^{\alpha}_{\rm NL}(\chi_E, \eta_k) \di x \right | 
\\
& \le 2 \mu_{n,\alpha} |E| \|\nabla^{\alpha}_{\NL}(\chi_E, \eta_k)\|_{L^\infty(\R^n; \,\R^n)} 
\\
& \le 2 \mu_{n,\alpha} |E| [\eta_k]_{B^\alpha_{\infty, 1}(\R^n)} 
\\
& = 
2 \mu_{n,\alpha} |E| [\eta]_{B^\alpha_{\infty, 1}(\R^n)} k^{-\alpha} \to 0 \ \text{ as } k \to + \infty.
\end{align*}
Thus, by the Dominated Convergence Theorem, we obtain
\begin{equation*}
{\rm e}_j \cdot D^\alpha_{\NL}(\chi_E, \chi_E)(\R^n) 
= 
\lim_{k \to + \infty} \int_{\R^n} \eta_k {\rm e}_j \cdot \di D^\alpha_{\NL}(\chi_E, \chi_E) = 0 \ \text{ for all } j \in \{1, \dots, n\},
\end{equation*}
which implies $D^\alpha_{NL}(\chi_E, \chi_E)(\R^n) = 0$. Consequently, since $D^\alpha \chi_E(\R^n) = 0$ by~\cite{Comi-Stefani22-L}*{Lem.~2.5}, we can integrate \eqref{eq:NL_grad_chi_E_repr} over $\R^n$ to get
\begin{equation*}
0 = D^\alpha_{\NL}(\chi_E, \chi_E)(\R^n) = \int_{\R^n} (1 - 2\,\overline{\chi_E}) \di D^\alpha \chi_E = - 2 \int_{\R^n} \overline{\chi_E} \di D^\alpha \chi_E,
\end{equation*}
proving~\eqref{eq:zero_measure_R_n_NL} and ending the proof.
\qed

\subsection{Proof of \texorpdfstring{\cref{res:convergence}}{Theorem 1.6}}

We prove~\eqref{item:weak_conv_blow_up}, \eqref{item:weak_conv_blow_up_tot_var} and~\eqref{item:conv_blow_up_tot_var_ball} separately.

\vspace*{1ex}

\textit{Proof of~\eqref{item:weak_conv_blow_up}}. 
Up to extracting a subsequence, we can also assume that $\chi_{\frac{E - x}{r_{k}}} \to \chi_F$ a.e.~in $\R^{n}$ as $r_k\to0^+$.
Hence, by the Dominated Convergence Theorem, we get
\begin{equation*}
\int_{\R^{n}} \chi_{\frac{E - x}{r_{k}}} \, \div^{\alpha}\varphi \di x 
\to 
\int_{\R^{n}} \chi_{F} \, \div^{\alpha} \varphi \di x
\quad
\text{as}\ r_k\to+\infty
\end{equation*}
for any $\varphi \in \Lip_{c}(\R^{n}; \R^{n})$, since $\div^{\alpha} \phi \in L^{1}(\R^{n})$ by \cite{Comi-Stefani19}*{Cor.~2.3}.
Thanks to the density of $\Lip_{c}(\R^{n}; \R^{n})$ into $C_{c}(\R^{n}; \R^{n})$ with respect to the uniform convergence, we infer~\eqref{item:weak_conv_blow_up}.

\vspace*{1ex}

\textit{Proof of~\eqref{item:weak_conv_blow_up_tot_var}}.
Given $\varphi \in C_{c}(\R^{n})$, we can estimate 
\begin{equation}
\label{eq:ciccia}
\begin{split}
\bigg| 
\int_{\R^{n}} \phi \di |D^{\alpha} \chi_{\frac{E - x}{r_{k}}}| 
&- 
\int_{\R^{n}} \phi \di |D^{\alpha} \chi_{F}| 
\bigg| 
\le 
{r_{k}^{\alpha-n}}
\int_{\R^{n}} \left|\varphi\left ( \tfrac{y - x}{r_{k}} \right )\right| 
|\nu_{E}^{\alpha}(y) - \nu_{E}^{\alpha}(x) | \di |D^{\alpha} \chi_{E}|(y)  
\\
& + 
\left|
\int_{\R^{n}} \varphi \big ( \nu_{E}^{\alpha}(x)\cdot \di D^{\alpha} \chi_{\frac{E- x}{r_{k}}} - \nu_{F}^{\alpha} \cdot \di D^{\alpha} \chi_{F} \big )
\right|
\end{split}
\end{equation}
On the one side, since $x \in \redb^{\alpha} E$, by~\cite{Comi-Stefani19}*{Th.~5.3} there are $A_{n, \alpha} > 0$ and $r_{x} > 0$ such that
\begin{equation}
\label{eq:ciop}
|D^{\alpha} \chi_{E}|(B_{r}(x)) \le A_{n, \alpha} r^{n - \alpha}
\quad
\text{for all}\
r \in (0, r_{x}).
\end{equation}
Hence, letting $R>0$ be such that $\supp\varphi \subset B_{R}$, we can exploit~\eqref{eq:ciop} to estimate the first term in the right-hand side of~\eqref{eq:ciccia} as
\begin{equation}
\label{eq:spesa}
\begin{split}
{r_{k}^{\alpha-n}}
\int_{\R^{n}} &\left|\varphi\left ( \tfrac{y - x}{r_{k}} \right )\right| 
|\nu_{E}^{\alpha}(y) - \nu_{E}^{\alpha}(x) | \di |D^{\alpha} \chi_{E}|(y)  
\\
&\le
C_{n, \alpha, R} 
\,
\mint{-}_{B_{r_{k} R}(x)} \left | \varphi\left(\tfrac{y - x}{r_{k}} \right ) \right | |\nu_{E}^{\alpha}(y) - \nu_{E}^{\alpha}(x)| \di |D^{\alpha} \chi_{E}|(y)
\\
& \le \|\varphi\|_{L^{\infty}(\R^{n})} \left ( \mint{-}_{B_{r_{k} R}(x)} |\nu_{E}^{\alpha}(y) - \nu_{E}^{\alpha}(x)|^{2} \di |D^{\alpha} \chi_{E}|(y) \right )^{\frac{1}{2}}\\
& =  \|\varphi\|_{L^{\infty}(\R^{n})} \left ( \mint{-}_{B_{r_{k} R}(x)} 2( 1 - \nu_{E}^{\alpha}(y) \cdot \nu_{E}^{\alpha}(x)) \di |D^{\alpha} \chi_{E}|(y) \right )^{\frac{1}{2}}
\end{split}
\end{equation} 
for all $r_k>0$ sufficiently small by Jensen's inequality, where $C_{n,\alpha,R}>0$ does not depend on~$k$.
Again since $x\in\redb^\alpha E$, we have 
\begin{equation} \label{eq:cip}
\nu_{E}^{\alpha}(x)
=
\lim_{r \to 0^+}
\mint{-}_{B_{r}(x)} \nu_{E}^{\alpha}(y) \di |D^{\alpha} \chi_{E}|(y).
\end{equation}
Therefore, by combining~\eqref{eq:spesa} with~\eqref{eq:cip}, we conclude that 
\begin{equation}
\label{eq:ping}
{r_{k}^{\alpha-n}}
\int_{\R^{n}} \left|\varphi\left ( \tfrac{y - x}{r_{k}} \right )\right| 
|\nu_{E}^{\alpha}(y) - \nu_{E}^{\alpha}(x) | \di |D^{\alpha} \chi_{E}|(y)
\to 0
\quad
\text{as}\ r_k\to0^+. 
\end{equation}
On the other side, by \cite{Comi-Stefani19}*{Prop.~5.9}, we have $\nu_{F}^{\alpha} = \nu_{E}^{\alpha}(x)$ $|D^{\alpha} \chi_{F}|$-a.e.\ in~$\R^{n}$.
Hence, in virtue of~\eqref{item:weak_conv_blow_up}, the second term in the right-hand side of~\eqref{eq:ciccia} satisfies
\begin{equation}
\label{eq:pong}
\begin{split}
\bigg|
\int_{\R^{n}} \varphi(y) \big (& \nu_{E}^{\alpha}(x)\cdot \di D^{\alpha} \chi_{\frac{E-x}{r_{k}}}(y) - \nu_{F}^{\alpha}(y) \cdot \di D^{\alpha} \chi_{F}(y) \big )
\bigg|
\\
&=
\left | \int_{\R^{n}}\varphi\, \nu_{E}^{\alpha}(x) \cdot \big( \di D^{\alpha} \chi_{\frac{E- x}{r_{k}}} - \di D^{\alpha} \chi_{F} \big) \right |
\to0^+
\quad
\text{as}\ r_k\to0^+
\end{split}
\end{equation} 
possibly passing to a further subsequence. 
Thus~\eqref{item:weak_conv_blow_up_tot_var} follows from~\eqref{eq:ciccia}, \eqref{eq:ping} and~\eqref{eq:pong}.

\vspace*{1ex}

\textit{Proof of~\eqref{item:conv_blow_up_tot_var_ball}}.
Points~\eqref{item:weak_conv_blow_up} and~\eqref{item:weak_conv_blow_up_tot_var} implies~\eqref{item:conv_blow_up_tot_var_ball} for all $R > 0$ such that $|D^{\alpha} \chi_{F}|(\partial B_{R}) = 0$.
Since $|D^{\alpha} \chi_{F}|(\partial B_{R}) = 0$ for a.e.\ $R > 0$ 
(by~\cite{AFP00}*{Exam.~1.63} for instance), 
the validity of~\eqref{item:conv_blow_up_tot_var_ball} immediately follows.
\qed

\subsection{Proof of \texorpdfstring{\cref{res:null_derivativel}}{Proposition 1.8}}

Assume $D^\alpha_i f=0$. Let $\phi\in C^\infty_c(\R^n)$ and set $\psi=(-\Delta)^{\frac{1-\alpha}2}\phi$. 
By~\cite{Comi-Stefani19}*{Lem.~3.28(ii)}, we have $\psi\in S^{\alpha,1}(\R^n)$ (i.e., $\psi\in BV^{\alpha,1}(\R^n)$ with $|D^\alpha\psi|\ll\Leb{n}$, see~\cite{Comi-Stefani19}*{Sec.~3.9} for an account) with $\nabla^\alpha\psi=\nabla\phi$. 
Hence, by~\cite{Comi-Stefani19}*{Th.~3.23}, we can find $(\psi_k)_{k\in\N}\subset C^\infty_c(\R^n)$ such that $\psi_k\to\psi$ in~$S^{\alpha,1}(\R^n)$ as $k\to+\infty$. 
Thus, we have
\begin{equation*}
\int_{\R^n}f\,\de_i\phi\di x
=
\int_{\R^n}f\,\de^\alpha_i\psi\di x
=
\lim_{k\to+\infty}
\int_{\R^n}f\,\de^\alpha_i\psi_k\di x
=
\lim_{k\to+\infty}
\int_{\R^n}\psi_k\,\di D_i^\alpha f=0,
\end{equation*} 
from which we readily get $D_if=0$. 
Viceversa, assume $D_if=0$ and suppose $f\in\Lip_b(\R^n)$ at first.
Then we can write $f(x)=g(\hat x_i)$ for all $x\in\R^n$ for some $g\in\Lip_b(\R^{n-1})$. 
Thus, by~\cite{Comi-Stefani22-A}*{Lem.~2.3}, we can compute
\begin{equation*}
\begin{split}
\nabla^{\alpha}_i f(x) 
&= 
\mu_{n, \alpha} \lim_{\eps \to 0^+} \int_{\{ |y| > \eps \}} \frac{y_if(x + y)}{|y|^{n + \alpha + 1}} 
\di y
\\
&=
\mu_{n, \alpha} 
\lim_{\eps \to 0} 
\int_{\R^{n-1}} g(\hat{x}_i + \hat{y}_i) \int_{\set*{ |y_i|^2 > (\eps^2 - |\hat{y}_i|^2)^+}} \frac{y_i}{(y_i^2 + |\hat{y}_i|^2)^{\frac{n + \alpha + 1}{2}}} \di y_i \di \hat{y}_i = 0
\end{split}
\end{equation*}
for all $x\in\R^n$, so that $D^\alpha_i f=0$. 
If now $f\in BV^{\alpha,\infty}_\loc(\R^n)$, then $f_\eps=\rho_\eps*f\in\Lip_b(\R^n)$ for all $\eps>0$. 
Since $D_if=0$ by assumption, also $D_if_\eps=\rho_\eps*D_if=0$ for all $\eps>0$, and thus
\begin{equation*}
\int_{\R^n}\phi\di D^\alpha_if
=
\lim_{\eps\to0^+}
\int_{\R^n}\rho_\eps*\phi\di D^\alpha_if
=
\lim_{\eps\to0^+}
\int_{\R^n}\phi\di D^\alpha_if_\eps
=0
\end{equation*}
by the Dominated Convergence Theorem for all $\phi\in C^\infty_c(\R^n)$, so that $D^\alpha_i f=0$.
\qed

\subsection{Proof of \texorpdfstring{\cref{res:splitting}}{Proposition 1.9}}

Since $D_1f=0$, there exists $g\in L^\infty(\R^{n-1})$ such that $f((t,x))=g(x)$ for a.e.\ $t\in\R$ and a.e.\ $x\in\R^{n-1}$.
Now let us assume that $f\in \Lip_b(\R^n)$ at first, so that also $g\in\Lip_b(\R^{n-1})$.
By~\cite{Comi-Stefani22-A}*{Lem.~2.3}, we can write 
\begin{align*}
(\nabla^{\alpha}_{\R^{n-1}})_i g(x)
& = 
\mu_{n-1, \alpha} 
\int_{\R^{n-1}}\frac{(g(z) - g(x))\,(z_i-x_i)}{|z - x|^{n + \alpha}} \di z 
\\
& = 
\mu_{n, \alpha} 
\,\frac{\Gamma \left ( \frac{n+\alpha}{2} \right ) \sqrt{\pi}}{\Gamma \left ( \frac{n+\alpha+1}{2} \right)} 
\int_{\R^{n-1}}\frac{(g(z) - g(x))\,(z_i-x_i)}{|z - x|^{n + \alpha}} \di z 
\\
& = 
\mu_{n, \alpha} 
\int_{\R^{n-1}} (g(z) - g(x))\,(z_i-x_i) \int_{\R} \frac{1}{( t^2 + |z - x|^2)^{\frac{n + \alpha + 1}{2}}}\di t \di z
\\
& = 
\mu_{n, \alpha} 
\int_{\R^{n-1}} \int_{\R} \frac{(f((t+s,z)) - f((s,x)))\,(z_i-x_i)}{( t^2 + |z - x|^2)^{\frac{n + \alpha + 1}{2}}}\di t \di z
\\
& = 
\mu_{n, \alpha} 
\int_{\R^{n}} \frac{(f(y) - f((s,x)))\,(y_i-(s,x)_i)}{|y-(s,x)|^{n + \alpha + 1}}\di y
=
(\nabla^{\alpha}_{\R^n})_i f((s,x))
\end{align*}
for all $x\in\R^{n-1}$, $s\in\R$ and $i=2,\dots,n$, where $\nabla^\alpha_{\R^m}$ denotes the operator~\eqref{eq:def_nabla_alpha} taken in the ambient space~$\R^m$, $m\in\N$.
In the above chain of equalities, we exploited the fact that, by the properties of the Gamma and Beta functions, for $u=|z-x|$ and $s=n+\alpha$, 
\begin{equation*}
u^s\int_{\R} \frac{\di t}{(t^2 + u^2)^{\frac{s + 1}{2}}}
=
2
\int_0^{+\infty}\frac{\di t}{(t^2 + 1)^{\frac{s + 1}{2}}}
\overset{[1+t^2=\frac1r]}{=}
\int_0^1 r^{\frac s2-1}(1-r)^{\frac12-1}\di r
=
\frac{\Gamma\left(\frac s2\right)\,\sqrt\pi}{\Gamma\left(\frac{s+1}2\right)}.
\end{equation*}
By \cref{res:null_derivativel}, $D^\alpha_{1, \R^n} f = 0$, and so $D_1\nabla^{\alpha}_{\R^n} f=0$.
Now let $f\in BV^{\alpha,\infty}_\loc(\R^n)$ and set $f_\eps=\rho_\eps*f$ for all $\eps>0$.
Then $f_\eps\in\Lip_b(\R^n)$ with $\nabla^\alpha_{\R^n}f_\eps=\rho_\eps*D^\alpha_{\R^n}f$ for all $\eps>0$.
Since $D_1f_\eps=0$, there is  $g_\eps\in\Lip_b(\R^{n-1})$ such that $f_\eps((t,x))=g_\eps(x)$ for all $t\in\R$ and $x\in\R^{n-1}$ and $g_\eps\to g$ a.e.\ in~$\R^{n-1}$ as $\eps\to0^+$.
Thus
\begin{equation*}
(\nabla^{\alpha}_{\R^{n-1}})_i g_\eps(x)
=
(\nabla^{\alpha}_{\R^n})_i f_\eps((s,x))
\end{equation*}
for all $x\in\R^{n-1}$, $s\in\R$, $i\in\set*{2,\dots,n}$ and $\eps>0$, thanks to \cref{res:smoothing}.
Note that $D_1D^\alpha_{\R^n} f=0$.
Indeed, since $D_1\nabla^\alpha_{\R^n} f_\eps=0$ for all $\eps>0$, we have 
\begin{equation*}
0=
\lim_{\eps\to0^+}
\int_{\R^n}D_1\psi\,\nabla^\alpha_{\R^n}f_\eps\di x
=
\lim_{\eps\to0^+}
\int_{\R^n}\rho_\eps*D_1\psi\,\di D^\alpha_{\R^n}f
=
\int_{\R^n}D_1\psi\,\di D^\alpha_{\R^n}f
\end{equation*}
for all $\psi\in C^\infty_c(\R^n)$ by the Dominated Convergence Theorem.
Now, given $\phi\in C^\infty_c(\R^{n-1})$ and $\sigma\in C^\infty_c(\R)$, setting $\psi(y)=\sigma(y_1)\,\phi(\hat y_1)$ for all $y\in\R^n$, we have $\psi\in C^\infty_c(\R^n)$ and so
\begin{align*}
\left(\int_\R\sigma\di t\right) \int_{\R^{n-1}}\phi\,(\nabla^\alpha_{\R^{n-1}})_ig_\eps\di x
&=
\left(\int_\R\sigma\di t\right) \int_{\R^{n-1}}\phi(x)\,(\nabla^\alpha_{\R^{n}})_if_\eps((0,x))\di x
\\
&=
\int_\R\sigma(t)\int_{\R^{n-1}}\phi(x)\,(\nabla^\alpha_{\R^{n}})_if_\eps((0,x))\di x\di t
\\
&=
\int_\R\sigma(t)\int_{\R^{n-1}}\phi(x)\,(\nabla^\alpha_{\R^{n}})_if_\eps((t,x))\di x\di t
\\
&=
\int_{\R^{n}}\sigma(y_1)\,\phi(\hat y_1)\,(\nabla^\alpha_{\R^{n}})_if_\eps(y)\di y
\\
&=
\int_{\R^{n}}\psi\,(\nabla^\alpha_{\R^{n}})_if_\eps\di y
\end{align*}
for all $\eps>0$ and $i\in\set*{2,\dots,n}$.
By the Dominated Convergence Theorem, we have
\begin{equation*}
\lim_{\eps\to0^+}
\int_{\R^{n-1}}\phi\,(\nabla^\alpha_{\R^{n-1}})_ig_\eps\di x
=
-
\lim_{\eps\to0^+}
\int_{\R^{n-1}}g_\eps\,(\nabla^\alpha_{\R^{n-1}})_i\phi\di x
=
-
\int_{\R^{n-1}}g\,(\nabla^\alpha_{\R^{n-1}})_i\phi\di x
\end{equation*}
and, similarly,
\begin{equation*}
\lim_{\eps\to0^+}
\int_{\R^{n}}\psi\,(\nabla^\alpha_{\R^{n}})_if_\eps\di y
=
\int_{\R^{n}}\psi\di(D^\alpha_{\R^{n}})_if.
\end{equation*}
We thus conclude that 
\begin{equation*}
- \left(\int_\R\sigma\di t\right) \int_{\R^{n-1}}g\,(\nabla^\alpha_{\R^{n-1}})_i\phi\di x
=
\int_{\R^{n}}\sigma(y_1)\,\phi(\hat y_1)\di(D^\alpha_{\R^{n}})_if(y)
\end{equation*}
for all $i\in\set*{2,\dots,n}$, $\phi\in C^\infty_c(\R^n)$ and $\sigma\in C^\infty_c(\R)$.
Hence $g\in BV^{\alpha,\infty}_\loc(\R^{n-1})$, with
\begin{equation*}
\int_{\R^{n}}\sigma(y_1)\,\phi(\hat y_1)\di(D^\alpha_{\R^{n}})_if(y)
=
\left(\int_\R\sigma\di t\right)
\left(\int_{\R^{n-1}}\phi\di (D^\alpha_{\R^{n-1}})_ig\right)
\end{equation*} 
for all $i\in\set*{2,\dots,n}$, $\phi\in C^\infty_c(\R^n)$ and $\sigma\in C^\infty_c(\R)$, yielding the conclusion. 
\qed

\subsection{Proof of \texorpdfstring{\cref{res:blowup}}{Theorem 1.7}}

By~\cite{Comi-Stefani19}*{Prop.~5.9}, we have $\chi_F\in BV^{\alpha,\infty}_\loc(\R^n)$ with $\nu^\alpha_F=\e_n$ $|D^\alpha\chi_F|$-a.e.\ in~$\R^n$.
Hence $D^\alpha_i\chi_F=0$ for all $i=1,\dots,n-1$ and $D^\alpha_n\chi_F\ge0$.
By \cref{res:null_derivativel}, we infer that also $D_i\chi_F=0$ for all $i=1,\dots,n-1$. 
Consequently,  $F=\R^{n-1}\times M$ for some measurable $M\subset\R$. 
We now prove the properties of the set~$M$.

\vspace*{1ex}

\textit{Proof of~\eqref{item:blowup_M_BV}}.
By repeatedly applying \cref{res:splitting}, we get $\chi_M\in BV^{\alpha,\infty}_{\loc}(\R)$, with
\begin{equation*}
(D^\alpha_{\R^n})_n\chi_F
=
\Leb{n-1}\otimes D^\alpha_{\R}\chi_M
=
\Leb{n-1}\otimes \de^\alpha\chi_M
\quad
\text{in}\
\M_{\loc}(\R^n),
\end{equation*}
from which we readily deduce that $\de^\alpha\chi_M\ge 0$. 

\vspace*{1ex}

\textit{Proof of~\eqref{item:blowup_M_meas_infty}}.
If $|M|\in(0,+\infty)$ by contradiction, then $\chi_M\in BV^{\alpha,1}_{\loc}(\R)$ and thus we obtain $u=I_{1-\alpha}\chi_M\in BV_\loc(\R)$ with $\de u=\de^\alpha\chi_M$, arguing exactly as in~\cite{Comi-Stefani19}*{Lem.~3.28(i)}.
Hence $\de u\ge0$ and thus $u$ is a non-negative and non-decreasing function. 
Moreover, since $|M|<+\infty$, we have $u\in L^p(\R)$ for all $p\in\left(\frac{1}{\alpha},+\infty\right)$, which immediately yields $u\equiv0$, so that $|M|=0$, a contradiction.
Hence $|M|\in\set*{0,+\infty}$ and, since $F^c \in \tang(E^c,x)$, we also get that $|M^c|\in\set*{0,+\infty}$ by a symmetrical argument.

\vspace*{1ex}

\textit{Proof of~\eqref{item:blowup_M_esssup}}.
Let $|M| = +\infty$ and assume $b=\esssup M<+\infty$ by contradiction.
Let $I_b=(b,b+1)$. 
By~\eqref{item:blowup_M_BV} and~\eqref{eq:GG_fract}, we can compute
\begin{equation*}
0\le \de^\alpha\chi_M(I_b)
=
-\int_M\nabla^\alpha\chi_{I_b}\di t.
\end{equation*}
By~\cite{Comi-Stefani19}*{Exam.~4.11}, $\nabla^\alpha\chi_{I_b}(t) > 0$ for all $t<b$, forcing $|M|=0$, which is a contradiction.

\vspace*{1ex}
\textit{Proof of \eqref{item:blowup_M_intervals}}.  
Let $M\ne\emptyset,\R$ be such that $P(M)<+\infty$.
Then, up to negligible sets,  $M=\cup_{k=1}^N I_k$ for $N\in\N$ closed intervals $I_k\subset\R$ with $a_k=\inf I_k<\sup I_k=b_k$, $a_k,b_k\in[-\infty,+\infty]$ and $\sup I_k<\inf I_{k+1}$ for all $k=1,\dots,N-1$.
Let us assume that $N\ge2$.
Since $|M|=+\infty$ by~\eqref{item:blowup_M_meas_infty}, we must have $b_N=+\infty$ by~\eqref{item:blowup_M_esssup}.
Since also $|M^c|=+\infty$ by~\eqref{item:blowup_M_meas_infty}, we must have $a_1>-\infty$.
In particular, $I_k$ is a compact interval for all $k=1,\dots,N-1$. 
By linearity and in virtue of~\cite{Comi-Stefani19}*{Exam.~4.11}, we have $\de^\alpha\chi_M=\nabla^\alpha\chi_M\,\Leb{1}$, with
\begin{equation*}
\nabla^\alpha\chi_M(t)
=
\sum_{k=1}^N
\nabla^\alpha\chi_{I_k}
=
c_\alpha\sum_{k=1}^{N-1}\left(|t-a_k|^{-\alpha}-|t-b_k|^{-\alpha}\right)
+
c_\alpha|t-a_N|^{-\alpha}
\end{equation*}
for all $t\in\R$ with $t\ne a_1,b_1,\dots,a_{N-1},b_{N-1},a_N$, where $c_\alpha>0$ depends on~$\alpha$ only. 
Since $N\ge2$, we have $\nabla^\alpha\chi_M(t)<0$ in an open neighborhood of $b_1$,
contradicting~\eqref{item:blowup_M_BV}.
We thus must have $N=1$ and so $M=(a_1,+\infty)$ for some  $a_1\in\R$, concluding the proof.
\qed

\subsection{Proof of \texorpdfstring{\cref{res:supp_sing}}{Theorem 1.10}}

Let $\chi_F\in W^{\alpha,1}(\R^n)$. 
By \cref{res:cap_W}, we have $\chi_{E\cap F}\in BV^{\alpha,\infty}_\loc(\R^n)$ with
$D^\alpha_\sing\chi_{E\cap F}=\chi_{F^1} D^\alpha_\sing\chi_E$
in $\M_\loc(\R^n;\R^n)$.
If $|E \cap F| = |F|$, then $\chi_{E \cap F} = \chi_{F}$ and so 
$
\chi_{F^1} |D^\alpha_\sing \chi_E|
=
|D^\alpha_\sing \chi_{E\cap F}| 
= 
|D^\alpha_\sing \chi_{F}| 
= 
0$.
If instead $|E \cap F| = 0$, then $\chi_{E\cap F} = 0$ and so again 
$
\chi_{F^1} |D^\alpha_\sing \chi_E|
=
|D^\alpha_\sing \chi_{E \cap F}| 
=0
$.
Therefore $|D^\alpha_\sing\chi_E|(F^1)=0$ whenever $|E^c\cap F| = 0$ or $|E \cap F| = 0$. Taking $F = B_r(x)$ for $x \in \R^n$ and $r > 0$, we get $\supp|D^\alpha_\sing\chi_E|\subset\de^- E$.
\qed

\subsection{Proof of \texorpdfstring{\cref{res:eff_bello}}{Theorem 1.12}}

We prove each statement separately.

\vspace*{1ex}

\textit{Proof of \eqref{item:eff_BV_frac}}.
Without loss of generality, we may assume that $x=0$ and $\nu^\alpha_E(0)=\e_n$.
By~\cite{Leonardi00}*{Prop.~2.1}, we must have $F=\lambda F$ for all $\lambda>0$. 
Due to \cref{res:blowup}, this implies that $F=\R^{n-1}\times M$ for some $M\subset\R$ such that $M=\lambda M$ for all $\lambda>0$.
Since $0\in\redb^\alpha_\eff E$, we must have $M\ne\emptyset,\R$.
As a consequence, $|M|=+\infty$ by \cref{res:blowup}\eqref{item:blowup_M_meas_infty}.
It is now plain to see that either $M=(0,+\infty)$ or $M=(-\infty,0)$, but the latter case is automatically excluded by points~\eqref{item:blowup_M_esssup} and~\eqref{item:blowup_M_intervals} of \cref{res:blowup}.  
We thus get that $F=H^+_{\e_n}(0)$, as claimed.
The remaining part of the statement is a simple application of~\cite{DelNin21}*{Def.~1.1 and Th.~1.2}.

\vspace*{1ex}

\textit{Proof of \eqref{item:eff_W}}.
By~\cite{Ponce-Spector20}*{Prop.~3.1}, we know that $\Haus{n - \alpha}(\partial^{*} E) = 0$. 
Since $\redb^\alpha_\eff E\subset\de^*E$ by \cref{def:eff}, we thus get that  $\Haus{n - \alpha}(\redb^\alpha_\eff E) = 0$ as well.

\vspace*{1ex}

\textit{Proof of \eqref{item:eff_BV}}.
Let $x \in \redb E$. 
By~\eqref{item:conv_blow_up_tot_var_ball} in \cref{res:convergence} and by \cref{res:halfspace}, we have
\begin{equation*}
\lim_{r \to 0^+} \frac{|D^{\alpha} \chi_{E}|(B_{rR}(x))}{(rR)^{n - \alpha}} 
= 
\lim_{r \to 0^+} 
|D^{\alpha} \chi_{\frac{E - x}{r}}|(B_{R}) 
= 
|D^{\alpha} \chi_{H^{+}_{\nu_{E}(x)}}|(B_{R})>0
\end{equation*}
for all $R > 0$. 
Hence, there exists $r_x >0$ such that
\begin{equation} \label{eq:non_zero_balls}
|D^{\alpha} \chi_{E}|(B_{\rho}(x)) > 0 
\quad 
\text{for any}\ 
\rho \in (0, r_x).
\end{equation}
Moreover, by~\eqref{item:weak_conv_blow_up} and~\eqref{item:weak_conv_blow_up_tot_var} in \cref{res:convergence} and by \cref{res:halfspace}, we have 
\begin{equation} \label{eq:fract_normal_classical_limit}
 \lim_{r \to 0^+} \frac{D^{\alpha} \chi_{E}(B_{r}(x))}{|D^{\alpha} \chi_{E}|(B_{r}(x))} =  \frac{ \displaystyle \int_{B_{1}} \nabla^{\alpha} \chi_{H^{+}_{\nu_{E}(x)}} \di y}{ \displaystyle \int_{B_{1}} |\nabla^{\alpha} \chi_{H^{+}_{\nu_{E}(x)}}| \di y} = \nu_{E}(x).
\end{equation}
The limits in~\eqref{eq:non_zero_balls} and~\eqref{eq:fract_normal_classical_limit} thus imply that $x \in \redb^{\alpha} E$ with $\nu^\alpha_E=\nu_E$ on~$\redb E$.
Since $\redb E \subset \de^{*} E$, we get $\redb E \subset \redb^{\alpha}_\eff E$.
The conclusion thus follows by recalling that $\Haus{n-1}(\de^*E)<+\infty$ and $\Haus{n-1}(\de^*E\setminus\redb E)=0$, see~\cite{Maggi12}*{Th.~16.2} for instance.
\qed

\subsection{Proof of \texorpdfstring{
\cref{res:ball}}{Proposition 1.14}}

Without loss of generality, we can assume $x_0=0$ and $r=1$.
By~\cite{Comi-Stefani19}*{Th.~3.18, Eq.~(3.26)} (applied to $f=\chi_{B_1}$), we have
\begin{equation*} 
\nabla^{\alpha} \chi_{B_{1}}(x) = - \frac{\mu_{n, \alpha}}{n + \alpha - 1} \int_{\partial B_{1}} \frac{y}{|x - y|^{n + \alpha - 1}} \di \Haus{n - 1}(y)
\end{equation*}
for all $x\in\R^n\setminus\de B_1$.
Changing variables, we easily get that 
\begin{align*} 
\int_{\partial B_{1}} \frac{y}{|x - y|^{n + \alpha - 1}} \di \Haus{n - 1}(y) 
= 
\frac{x}{|x|} \int_{\partial B_{1}} \frac{y_{1}}{||x| \e_{1} - y|^{n + \alpha - 1}} \di \Haus{n - 1}(y)
\end{align*}
for all $x\in\R^n$ with $|x|\ne0,1$, 
since
$ \displaystyle
\int_{\partial B_{1}} \frac{y_{i}}{||x| \e_{1} - y|^{n + \alpha - 1}} \di  \Haus{n - 1}(y) = 0 
$
for $i \in \{2, \dots, n\}$ by symmetry.
We also notice that 
\begin{align*} 
& \int_{\partial B_{1}} \frac{y_{1}}{||x| \e_{1} - y|^{n + \alpha - 1}} \di  \Haus{n - 1}(y) \\
& = 
\int_{\partial B_{1} \cap \{ y_{1} > 0 \}} y_{1} \left ( \tfrac{1}{((|x| - y_{1})^{2} + |\hat{y}_{1}|^{2})^{\frac{n + \alpha - 1}{2}}} - \tfrac{1}{((|x| + y_{1})^{2} + |\hat{y}_{1}|^{2})^{\frac{n + \alpha - 1}{2}}} \right ) \di \Haus{n - 1}(y) > 0
\end{align*}
for all $x\in\R^n$ with $|x|\ne0,1$. 
Hence, we can write
\begin{equation} \label{eq:fract_grad_ball} 
\nabla^{\alpha} \chi_{B_{1}}(x) = - \frac{\mu_{n, \alpha}}{n + \alpha - 1}\, g_{n, \alpha}(|x|)\, \frac{x}{|x|}
\end{equation}
for all $x\in\R^n$ with $|x|\ne0,1$,
where
\begin{equation*}
g_{n, \alpha}(t) = \int_{\partial B_{1}} \frac{y_{1}}{|t \e_{1} - y|^{n + \alpha - 1}} \di  \Haus{n - 1}(y) > 0 
\end{equation*}
for all $t\ge 0$.
We now claim that
\begin{equation} \label{eq:fract_normal_ball}
\nu^{\alpha}_{B_{1}}(x) = - \frac{x}{|x|}
\quad
\text{for all}\ x\ne0.
\end{equation}
Indeed, since
\begin{equation*}
\left | \frac{\displaystyle \int_{B_{r}(x)} \left ( \frac{y}{|y|} - \frac{x}{|x|} \right ) g_{n, \alpha}(|y|) \di y}{\displaystyle \int_{B_{r}(x)} g_{n, \alpha}(|y|) \di y} \right | \le \sup_{y \in B_{r}(x)} \left | \frac{y}{|y|} - \frac{x}{|x|} \right |,
\end{equation*}
by~\eqref{eq:fract_grad_ball} we have
\begin{equation*}
\lim_{r \to 0^+} \frac{\displaystyle \int_{B_{r}(x)} \nabla^{\alpha} \chi_{B_{1}}(y) \di y}{\displaystyle \int_{B_{r}(x)} |\nabla^{\alpha} \chi_{B_{1}}(y)| \di y} 
= 
-
\lim_{r \to 0} \frac{\displaystyle \int_{B_{r}(x)} \frac{y}{|y|}\, g_{n, \alpha}(|y|) \di y}{\displaystyle \int_{B_{r}(x)} g_{n, \alpha}(|y|) \di y}
=-\frac x{|x|},
\end{equation*}
proving~\eqref{eq:fract_normal_ball}. 
If $x = 0$, then
$ \displaystyle
\int_{B_{r}} \frac{y}{|y|} \,g_{n, \alpha}(|y|) \di y = 0
$
by symmetry, so that
\begin{equation*}
\lim_{r \to 0^+} \frac{\displaystyle \int_{B_{r}} \nabla^{\alpha} \chi_{B_{1}}(y) \di y}{\displaystyle \int_{B_{r}} |\nabla^{\alpha} \chi_{B_{1}}(y)| \di y} 
= 
\lim_{r \to 0^+} \frac{\displaystyle \int_{B_{r}} \frac{y}{|y|} g_{n, \alpha}(|y|) \di y}{\displaystyle \int_{B_{r}} g_{n, \alpha}(|y|) \di y} = 0.
\end{equation*}
Consequently, $\redb^{\alpha} B_{1} = \R^{n} \setminus \{0\}$, 
$\redb^{\alpha}_\eff B_{1} = \partial B_{1} = \redb B_{1}$ and $\nu^\alpha_{B_1}=\nu_{B_{1}}$ on $\partial B_{1}$.
\qed


\end{document}